\documentclass[11pt]{article}
\usepackage[a4paper,margin=2.5cm]{geometry}
\usepackage[comma, round]{natbib}
\usepackage{float}
\usepackage{amssymb}
\usepackage{amsthm}
\usepackage{amsmath}
\usepackage{titletoc}
\usepackage{mathrsfs}
\usepackage[colorlinks, allcolors=blue]{hyperref}
\usepackage{amsfonts}
\usepackage{graphicx}
\usepackage[linesnumbered, ruled, vlined]{algorithm2e}
\usepackage{esint}
\usepackage{bbm}
\usepackage{mathtools}
\usepackage[shortlabels]{enumitem}
\usepackage{authblk}

\numberwithin{equation}{section}
\numberwithin{figure}{section}

\theoremstyle{plain}
\newtheorem{thm}{Theorem}[section]

\newtheorem{lem}[thm]{Lemma}
\newtheorem{prop}[thm]{Proposition}

\theoremstyle{definition}
\newtheorem{defn}{Definition}[section]
\newtheorem{asmp}{Assumption}[section]
\newtheorem{rmk}[defn]{Remark}
\newtheorem{exam}{Example}[section]
\newtheorem{prob}[exam]{Problem}

\newcommand{\1}{\mathbbm{1}}

\newcommand{\md}{\mathop{}\mathopen\mathrm{d}}
\newcommand{\E}{E}

\renewcommand{\P}{P}

\newcommand{\Id}{\operatorname{Id}}
\newcommand{\ot}{\operatorname{OT}}

\newcommand{\proj}{\operatorname{proj}}
\newcommand{\supp}{\operatorname{supp}}

\DeclareMathOperator*{\argmin}{arg\,min}
\DeclareMathOperator*{\arginf}{arg\,inf}

\def\lb{\mathopen{}\mathclose\bgroup\left}
\def\rb{\aftergroup\egroup\right}

\title{Schrödinger bridge with transport relaxation}
\author{Yifan Jiang\thanks{Email:  {\tt yifan.jiang@imperial.ac.uk}}}
\affil{Department of Mathematics, Imperial College London}
\author{Renyuan Xu\thanks{Email: {\tt renyuanxu@stanford.edu}}}
\affil{Department of Management Science and Engineering, Stanford University}
\author{Luhao Zhang\thanks{Email: {\tt luhao.zhang@jhu.edu}}}
\affil{Department of Applied Mathematics and Statistics, Johns Hopkins University }
\date{}

\begin{document}
\maketitle

\begin{abstract}
	Motivated by modern machine learning applications where we only have access to empirical measures constructed from finite  samples, we relax the marginal constraints of the classical Schrödinger bridge problem by penalizing the transport cost between the bridge’s marginals and the prescribed marginals.
	We derive a duality formula for this transport-relaxed bridge and demonstrate that it reduces to a finite-dimensional concave optimization problem when the prescribed marginals are discrete and the reference distribution is absolutely continuous.
	We establish the existence and uniqueness of solutions for both the primal and dual problems.
	Moreover, as the penalty blows up, we characterize the limiting bridge as the solution to a discrete Schrödinger bridge problem and identify a leading-order logarithmic divergence.
	Finally, we propose gradient ascent and Sinkhorn-type algorithms to numerically solve the transport-relaxed Schrödinger bridge, establishing a linear convergence rate for both algorithms.


\end{abstract}

\section{Introduction}
The Schrödinger bridge problem was originally introduced in \citet{schrodinger1931uber,schrodinger1932sur} as follows: given a Brownian diffusion  with law \(\gamma\), conditioned on a prescribed initial distribution \(\mu\in \mathscr{P}(\mathbb{R}^{d})\) and  a prescribed terminal distribution \(\nu\in \mathscr{P}(\mathbb{R}^{d})\), what is the most likely evolution of the diffusion connecting these marginals?
This question is a prototype of the large deviation principle.
Let \(\hat{\pi}_{n}\) be the empirical distribution of the observed  trajectories of {\color{blue}$n$ independent} particles, and
formally, Sanov's theorem states that for a regular collection of evolutions \(\mathcal{O}\subseteq \Pi(\mu,\nu)\),
\begin{equation*}
	\P(\hat{\pi}_{n}\in \mathcal{O})\simeq\sup_{\pi \in \mathcal{O}}\exp(-n H(\pi|\gamma)),
\end{equation*}
where the relative entropy \(H(\cdot|\gamma)\) acts as a rate function.
Consequently, the most likely evolution \(\pi^{*}\in \Pi(\mu,\nu)\), a coupling connecting \(\mu\)  and \(\nu\), is given by the optimization problem
\begin{equation}
	\label{eqn-sb}
	\pi^{*}=\arginf_{\pi\in \Pi(\mu,\nu)} H(\pi|\gamma).
\end{equation}
The Schrödinger bridge has since then found fruitful applications extending well beyond statistical physics.
In particular, its static formulation is closely connected to entropic optimal transport, which underpins the modern computational optimal transport methods \citep{nutz2022entropic,rigollet2018entropic,carlier2017convergence}.
More recently, the dynamic Schrödinger bridge has been recast as a stochastic control problem \citep{dai1991stochastic,chen2016relation}, attracting significant interests in generative AI.
In this framework, one typically begins with a tractable source distribution \(\mu\) and aims to learn an unknown target distribution \(\nu\) corresponding to an observed data set \citep{de2021diffusion,vargas2021solving,wang2021deep}.
The goal is to generate high-quality synthetic samples that closely resemble the target distribution by numerically approximating the solution of the Schrödinger bridge problem \eqref{eqn-sb}.

However,  significant gaps remain between the theoretical framework and practical implementations in these aforementioned works.
A necessary  and sufficient condition for the existence of the Schrödinger bridge is that
\begin{equation}
	\label{eqn-cond}
	\{\pi\in \Pi(\mu,\nu): H(\pi|\gamma)<\infty\} \text{ is not empty.}
\end{equation}
In many data generation tasks, this condition is violated: one typically has access only to a finite number of samples, resulting in discrete empirical measures \(\mu\) and \(\nu\).
Motivated by this practical limitation, we focus on a semi-discrete setting where \(\mu=\sum_{i}^{n}a_{i}\delta_{x_{i}}\) and \(\nu=\sum_{j=1}^{m}b_{j}\delta_{y_{j}}\) are discrete while the reference measure \(\gamma\in \mathscr{P}(\mathbb{R}^{d}\times \mathbb{R}^{d})\) is absolutely continuous with a density \(\rho\).

A natural and theoretically interesting question is  therefore whether one can construct an entropy-regularized coupling whose marginals are not enforced exactly, but are instead allowed to deviate from the prescribed measures under a suitable relaxation.
To address this ill-posedness and accommodate the semi-discrete setting regarding empirical data, we propose a \emph{transport-relaxed} Schrödinger bridge formulation:
\begin{equation}
	\label{eqn-tsb}
	\inf_{\pi\in \mathscr{P}(\mathbb{R}^{d}\times \mathbb{R}^{d})}I(\pi):=\inf_{\pi\in \mathscr{P}(\mathbb{R}^{d}\times \mathbb{R}^{d})}  \{\ot(\mu,\pi_{x})+\ot(\nu,\pi_{y}) +H(\pi | \gamma)\},
\end{equation}
where \(\ot\) denotes the transport cost induced by the squared norm \(\|\cdot-\cdot\|^{2}\) and \(\pi_{x}\), \(\pi_{y}\) are the marginals of \(\pi\).
Comparing to the classical Schrödinger bridge problem, we do not impose any constraints on the marginals of the coupling \(\pi\).
Instead, we interpret it as a \emph{free  boundary} problem, where the undetermined marginals provide a systematic approach to mollify the empirical distribution of the dataset. In particular, by appropriately choosing the marginal OT penalties, the marginals of the optimal solution remain close to the prescribed measures
\(\mu\) and \(\nu\).

\subsection{Main results}
Our first result establishes a duality formulation for \eqref{eqn-tsb}.
For any \(\alpha\in \mathbb{R}^{n}\) and \(\beta\in \mathbb{R}^{m}\), we define \(f(x,\alpha):=\sup_{1\leq i\leq n}\{\langle x, x_{i} \rangle+\alpha_{i}\}\) and \(g(y,\beta):=\sup_{1\leq j\leq m}\{\langle y, y_{j} \rangle+\beta_{j}\}\).
In Theorem \ref{thm-semi}, we prove a strong duality relation:
\begin{equation}
	\inf_{\pi\in \mathscr{P}(\mathbb{R}^{d}\times \mathbb{R}^{d})}I(\pi)=\sup_{(\alpha,\beta)\in \mathbb{R}^{n+m}} U(\alpha,\beta),
\end{equation}
where
\begin{equation*}
	U(\alpha,\beta)=\sum_{i=1}^{n}a_{i}(\alpha_{i}+\frac{1}{2}\|x_{i}\|^{2})+\sum_{j=1}^{m}b_{j}(\beta_{j}+\frac{1}{2}\|y_{j}\|^{2})-\int\exp(f(x,\alpha)+g(y,\beta)-\frac{1}{2}\|x\|^{2}-\frac{1}{2}\|y\|^{2})\md \gamma+1.
\end{equation*}
In particular, the dual problem reduces to a {\it finite-dimensional} optimization problem.
Furthermore, we establish the existence and uniqueness of solutions for both the primal and dual problems.
The uniqueness of the dual problem is understood modulo the equivalence relation \((\alpha,\beta)\sim_{\oplus}(\alpha',\beta')\), which holds if \(\alpha_{i}+\beta_{j}=\alpha_{i}'+\beta_{j}'\) for any \(1\leq i\leq n\) and \(1\leq j\leq m\).
We refer readers to Theorem \ref{thm-dual} for a general duality formula for \eqref{eqn-tsb}, allowing \(\mu\), \(\nu\), and \(\gamma\) on general Polish spaces.

We then introduce a parameterized transport-relaxed Schrödinger bridge:
\begin{equation}
	\label{eqn-tsb2}
	\inf_{\pi\in \mathscr{P}(\mathbb{R}^{d}\times \mathbb{R}^{d})}I^{\varepsilon}(\pi):=\inf_{\pi\in \mathscr{P}(\mathbb{R}^{d}\times \mathbb{R}^{d})}  \{\varepsilon^{-1}\ot(\mu,\pi_{x})+\varepsilon^{-1}\ot(\nu,\pi_{y}) +H(\pi | \gamma)\}.
\end{equation}
Let \(\pi^{*,\varepsilon}\) denote the unique optimizer of this problem.
As \(\varepsilon\) approaches 0, the transport relaxation becomes tighter and tighter, and it is expected to enforce a hard constraint at the limit.
We analyze the limiting behavior of \(\pi^{*,\varepsilon}\).
Under condition \eqref{eqn-cond}, we show in Proposition \ref{prop-blowup} that \(\pi^{*,\varepsilon}\) converges weakly to \(\pi^{*}\) the solution to the classical Schrödinger bridge.
Under a semi-discrete setting, however, condition \eqref{eqn-cond} no longer holds, leading to the blow-up
\begin{equation*}
	\lim_{\varepsilon\to 0}I^{\varepsilon}(\pi^{\varepsilon,*})=+\infty.
\end{equation*}
At first glance, it is not clear whether  hard constraints are recovered, i.e., whether
\[
	\lim_{\varepsilon\to 0}\ot(\mu,\pi_{x}^{*,\varepsilon})=\lim_{\varepsilon\to 0}\ot(\nu,\pi_{y}^{*,\varepsilon})=0.
\]
In Theorem \ref{thm-blowup}, we give an affirmative answer and establish the weak convergence of \(\pi^{*,\varepsilon}\) to a  limiting coupling \(\pi^{*,0}\in\Pi(\mu,\nu)\).
Moreover, \(\pi^{*,0}\) is the unique solution to a discrete Schrödinger bridge problem:
\begin{equation*}
	\pi^{*,0}=\arginf_{\pi\in\Pi(\mu,\nu)}H(\pi|\sigma),
\end{equation*}
where \(\sigma\) is a discrete positive measure given by  \(\sigma=\sum_{i=1}^{n}\sum_{j=1}^{m}(2 \pi)^{d}\rho(x_{i},y_{j})\delta_{(x_{i},y_{j})}\).
We also characterize the asymptotic behavior of \(I^{\varepsilon}(\pi^{*,\varepsilon})\) by
\begin{equation}
	I^{\varepsilon}(\pi^{*,\varepsilon}) = -d\ln \varepsilon + H(\pi^{*,0}|\sigma)+ o(1) \text{ as } \varepsilon\to 0.
\end{equation}
We emphasize that the leading-order  blow-up term  \(-d\ln \varepsilon\) is \emph{intrinsic}.
It depends only on the dimension of the underlying space and is independent of the choice of \(\mu\), \(\nu\), and \(\gamma\).

Built upon the dual formulation, we propose two numerical algorithms.
The first is a gradient ascent method on the dual objective where we update dual potentials as
\begin{equation*}
	(\alpha^{t+1},\beta^{t+1})^{\intercal}=(\alpha^{t},\beta^{t})^{\intercal} + \eta \nabla U(\alpha^{t},\beta^{t}).
\end{equation*}
Here, \(\eta>0\) is a fixed step size and the gradient term \(\nabla U (\alpha,\beta)\)  is explicitly given by
\begin{equation*}
	\nabla U(\alpha,\beta)=
	\left(
	\begin{aligned}
		a_{i}-\int_{A_{i}(\alpha)\times \mathbb{R}^{d}} \exp\Bigl(f(x,\alpha)+g(y,\beta)-\frac{1}{2}\|x\|^{2}-\frac{1}{2}\|y\|^{2}\Bigr)\gamma(\md x,\md y) \\
		b_{j}-\int_{\mathbb{R}^{d} \times B_{j}(\beta)} \exp\Bigl(f(x,\alpha)+g(y,\beta)-\frac{1}{2}\|x\|^{2}-\frac{1}{2}\|y\|^{2}\Bigr)\gamma(\md x,\md y)
	\end{aligned}
	\right),
\end{equation*}
where \(A_{i}\) and \(B_{j}\) are Laguerre cells given by
\begin{equation*}
	A_i(\alpha) :=\{x\in \mathbb{R}^{d}:f(x,\alpha)=\langle x,x_i\rangle +\alpha_i\} \text{ and } B_j(\beta) :=\{y\in \mathbb{R}^{d}:g(y,\beta)=\langle y,y_j\rangle +\beta_j\}.
\end{equation*}
Proposition \ref{prop-ga} establishes a linear convergence rate  in the \(l^{2}\) norm provided the step size \(\eta\) is chosen below a threshold \(\eta_{0}\).

The second approach is motivated by the Sinkhorn algorithm where we  view the first order optimality condition  as a fixed point problem.
Since \(\nabla U(\alpha,\beta)=0\) does not admit an explicit closed-form update, we introduce an approximate iteration scheme:
\begin{equation*}
	\left\{	\begin{aligned}
		\alpha_{i}^{t+1} & :=\frac{1}{\lambda}\log(a_{i})-\frac{1}{\lambda}\log\Biggl(\int \phi_{\lambda}(x,\alpha)^{1-\lambda} \psi_{\lambda}(y,\beta)  \exp\Bigl(\lambda \langle x, x_{i}\rangle -\frac{1}{2}\|x\|^{2}-\frac{1}{2}\|y\|^{2}\Bigr) \gamma(\md x,\md y)\Biggr), \\
		\beta_{j}^{t+1}  & :=\frac{1}{\lambda}\log(b_{j})-\frac{1}{\lambda}\log\Biggl(\int \phi_{\lambda}(x,\alpha) \psi_{\lambda}(y,\beta)^{1-\lambda}  \exp\Bigl(\lambda \langle y, y_{j} \rangle-\frac{1}{2}\|x\|^{2}-\frac{1}{2}\|y\|^{2}\Bigr) \gamma(\md x,\md y)\Biggr),
	\end{aligned}
	\right.
\end{equation*}
where \(\phi_{\lambda}(x,\alpha)=\bigl(\sum_{i=1}^{n}\exp(\lambda \langle x, x_{i} \rangle+\lambda \alpha_{i})\bigr)^{1/\lambda}\) and \(\psi_{\lambda}(y,\beta)=\bigl(\sum_{j=1}^{m}\exp(\lambda \langle y, y_{j} \rangle+\lambda \beta_{j})\bigr)^{1/\lambda}\).
Proposition \ref{prop-sinkhorn} shows that this mapping is a contraction in the  \(l^{\infty}\) norm, yielding linear convergence to a unique fixed point \((\alpha^{\lambda},\beta^{\lambda})\).
As \(\lambda\) goes to infinity, the limit of \((\alpha^{\lambda},\beta^{\lambda})\) satisfies the first order condition \(\nabla U(\alpha,\beta)\), and thus recovers the optimal dual potentials.

\subsection{Related literature}
We discuss the connections between the transport-relaxed Schrödinger bridge  and several strands of literature  on regularized optimal transport and relaxed Schrödinger bridge problem.
The setting closest to ours is that of \citet{garg24soft}, where the authors introduce a relaxed Schrödinger bridge by adding an entropic penalization on the marginal constraints.
It is later extended in \citet{ma2025schr} which allows a more general penalization beyond the relative entropy.
However, in all these works, even after relaxation, the marginals of the resulting bridge are still required to be absolutely continuous with respect to the marginals of the reference measure, which is inconsistent with many practical settings where the marginals are empirical measures and hence typically do not admit densities.
As a consequence, the semi-discrete setting considered in the present paper
, which closely reflects the data nature of modern machine learning problems, falls outside the scope of the existing theory.
Moreover, our quantitative asymptotic analysis in Section~\ref{sec-blowup} appears to be novel and is not covered by previous results on relaxed Schrödinger bridges.
We refer readers to \citet{leonard2013survey} for a more extensive review.

Our duality result in Theorem \ref{thm-dual} is closely linked to the well-known duality of entropic optimal transport; see for example \citet{peyre2019computational, cuturi2013sinkhorn,nutz2021introduction}.
Conceptually, the characterization of the limit Schrödinger bridge as the penalty blows up in Section \ref{sec-blowup} aligns with the entropic selection in entropic optimal transport; see \citet{marino18entropic,aryan25entropic,ley25entropic}.
Among other regularized optimal transport problems, the quadratically regularized optimal transport \citep{nutz25quadratically, gonzalezsanz2025linear} has recently gained an increasing attention due to its analytic tractability.
Although the transport-relaxed Schrödinger bridge differs in structure from quadratically regularized optimal transport, their dual problems share similar key geometric properties: neither is strictly concave, yet both are locally uniformly concave at the optimizer.
This observation underlies our convergence analysis of the gradient ascent algorithm in Section~\ref{sec-ga}, which is inspired by its quadratically regularized optimal transport counterpart in \citet{gonzalezsanz2025linear}.
We also remark that the unbalanced optimal transport problem \citep{sejourne23unblanced} is formulated in a similar form.
Compared to \eqref{eqn-tsb}, the relative entropy term and the optimal transport term swap their roles.

The adoption of transport relaxation/regularization has appeared extensively in the context of robust optimization. A prominent example is Wasserstein distributionally robust optimization, where the objective is penalized by a transport cost between the adversarial and reference models.
Duality results for this setting are established in \citet{blanchet19quantifying,zhang24short}, and the sensitivity with respect to the penalty strength is obtained in \citet{bartl21sensitivity}.
Furthermore, transport relaxation  extends naturally to dynamic settings with the notion of causal optimal transport.
We refer interested readers to \citet{jiang24Duality,jiang25sensitivity,sauldubois24first,bartl23sensitivity} for more details.
More broadly, optimal transport has emerged as a popular regularization tool in modern machine learning; see for example \citet{sinha18certifying,bai23wasserstein,bai25wasserstein}.

Finally, our work is related to applications of Schrödinger bridges in generative AI, which provide key motivation for the framework we study. This direction has grown rapidly in the recent generative modeling literature.
Representative methods approximate the classical iterative proportional fitting (IPF) procedure of \citet{deming1940least} using either score matching or maximum likelihood~\citep{de2021diffusion,vargas2021solving}, with extensions based on auxiliary bridges to address non-smooth target distributions~\citep{wang2021deep}.
Since then, a substantial body of work has further advanced Schrödinger bridge–based generative modeling, including
\citep{chen2021likelihood,peluchetti2023diffusion,hamdouche2023generative,shi2024diffusion,alouadi2026lightsbb}. More broadly speaking, the mathematics of Schrödinger bridges closely resembles the core idea behind diffusion models: generating samples by starting from a simple noise distribution and running a time-reversed diffusion process toward the data distribution. Representative diffusion-model works include \cite{ho2020denoising,song2020score,chen2023score,han2024neural}. We refer interested readers to \cite{chen2024overview,lai2025principles} for more details.

\subsection{Outline}
The rest of the paper is organized as follows.
In Section \ref{sec-pre}, we introduce necessary notations and recall basic properties of optimal transport and Schrödinger bridge.
In Section \ref{sec-tsb}, a general duality formulation is derived for \eqref{eqn-tsb} in Theorem \ref{thm-dual}.
Under a semi-discrete setting, the duality reduces to a finite-dimensional optimization problem.
In Theorem \ref{thm-semi}, we establish the existence and uniqueness of both the primal and the dual problem.
We investigate the limiting and asymptotic behavior of \eqref{eqn-tsb2}.
We characterize the limit of \(\pi^{*,\varepsilon}\) as the solution to a discrete Schrödinger bridge problem and derive the leading-order blow-up of \(I^{\varepsilon}(\pi^{*,\varepsilon})\).
In Sections \ref{sec-ga} and \ref{sec-sinkhorn}, we propose two numerical algorithms and obtain their linear convergence guarantees in Propositions \ref{prop-ga} and \ref{prop-sinkhorn}.
In Section \ref{sec-exam}, we conclude with numerical examples.

\section{Preliminaries}
\label{sec-pre}
\subsection{Notations}
For a  Polish space \(\mathcal{Z}\), we equip it with its Borel \(\sigma\)-algebra \(\mathcal{B}(\mathcal{Z})\).
Let \( \mathscr{P}(\mathcal{Z})\) be the space of Borel probability measures on \(\mathcal{Z}\)  and \(\mathscr{M}_{+}(\mathcal{Z})\) be the space of positive Borel measures.
For \( \sigma ,\gamma\in \mathscr{M}_{+}(\mathcal{Z})\), their relative entropy is given by
\begin{equation*}
	H(\sigma|\gamma):= \left\{
	\begin{aligned}
		 & \E_{\sigma}\Bigl[\log\Bigl(\frac{\md \sigma}{\md \gamma}\Bigr)\Bigr]  \text{ if } \sigma\ll \gamma, \\
		 & +\infty \hspace{1.5cm}                                                           \text{ otherwise.}
	\end{aligned}
	\right.
\end{equation*}
We denote the space of continuous functions on \(\mathcal{Z}\) by \(C(\mathcal{Z})\), the space of Borel measurable functions by \(L^{0}(\mathcal{Z})\), and the space of \(\eta\)-integrable functions by \(L^{1}(\mathcal{Z},\eta)=L^{1}(\eta)\).

Let \(\mathcal{X}\) and \(\mathcal{Y}\) be two Polish spaces.
For any \(\mu\in \mathscr{P}(\mathcal{X})\) and \(\nu\in \mathscr{P}(\mathcal{Y})\),  the set of couplings between \(\mu\) and \(\nu\) is given by
\begin{equation*}
	\Pi(\mu,\nu):=\{\pi\in \mathscr{P}(\mathcal{X}\times\mathcal{Y}):\pi(\cdot\times \mathcal{Y})=\mu(\cdot)\text{ and } \pi(\mathcal{X}\times\cdot)=\nu(\cdot)\}.
\end{equation*}
By \(\Pi(\mu,*)\) we denote the set of couplings with a given first marginal \(\mu\).
Given \(\mu\in \mathscr{P}(\mathcal{X})\) and a measurable map \(\Phi:\mathcal{X}\to\mathcal{Y}\), we define the \emph{pushforward} map \(\Phi_{\#}: \mathscr{P}(\mathcal{X})\to \mathscr{P}(\mathcal{Y})\) as
\begin{equation*}
	\Phi_{\#}\mu:=\mu\circ \Phi^{-1} \quad\text{for any } \mu\in \mathscr{P}(\mathcal{X}).
\end{equation*}
For any \(\pi\in \mathscr{P}(\mathcal{X}\times \mathcal{Y})\), we write \(\pi_{x}=[(x,y)\mapsto x]_{\#}\pi\) and \(\pi_{y}=[(x,y)\mapsto y]_{\#}\pi\) as the marginals of \(\pi\).
For functions \(f:\mathcal{X}\to \mathbb{R}\) and \(g: \mathcal{Y}\to \mathbb{R}\), we define \(f\oplus g:\mathcal{X} \times \mathcal{Y}\to \mathbb{R}\) as \(f\oplus g(x,y)= f(x)+g(y)\).

By \(\mathbf{0}\)(\(\mathbf{1}\)) we denote the vector with all zero(unit) entries.
For a vector \(v=(v_{1},\dots,v_{d})\in \mathbb{R}^{d}\), we write \(\overline{v}=\max_{i}v_{i}\), \(\underline{v}=\min_{i}v_{i}\).
Let \(\alpha^{i}\in \mathbb{R}^{n}\) and \(\beta^{i}\in \mathbb{R}^{m}\) for \(i=1,2\).
We say \((\alpha^{1},\beta^{1})\sim_{\oplus}(\alpha^{2},\beta^{2})\) if there exists \(r\in\mathbb{R}\) such that
\(\alpha^{1}=\alpha^{2}+r\mathbf{1}\) and \(\beta^{1}=\beta^{2}-r\mathbf{1}\).
On the quotient space \(\mathbb{R}^{n+m}_{\oplus}:=\mathbb{R}^{n+m}/_{\sim \oplus}\), we equip it with the \(l^{2}_{\oplus}\) norm and the \(l^{\infty}_{\oplus}\) norm given by
\begin{equation*}
	\|(\alpha,\beta)\|_{l^{2}_{\oplus}}=\|(\alpha_{i}+\beta_{j})_{1\leq i\leq n, 1\leq j\leq m}\|_{l^{2}} \text{ and }\|(\alpha,\beta)\|_{l^{\infty}_{\oplus}}=\|(\alpha_{i}+\beta_{j})_{1\leq i\leq n, 1\leq j\leq m}\|_{l^{\infty}}.
\end{equation*}
It is direct to verify both \(l^{2}_{\oplus}\) and \(l^{\infty}_{\oplus}\) are complete norms on \(\mathbb{R}^{n+m}_{\oplus}\), and they induce the same topology.
We denote the quotient map by \(\proj_{\oplus}:\mathbb{R}^{n+m}\to \mathbb{R}^{n+m}_{\oplus}\).

Throughout the paper, we follow the convention that \(\infty-\infty=-\infty\).
In particular, we define for any \(f\in L^{0}(\mathcal{X})\)
\begin{equation*}
	\E_{\mu}[f]=\E_{\mu}[f^{+}]-\E_{\mu}[f^{-}].
\end{equation*}
For any map \(\Phi:\mathcal{X} \to \mathcal{X}\), we write its \(t\)-th iteration as \(\Phi^{(t)}\).

\subsection{Optimal transport and Schrödinger bridge}
Let \(\mu\in \mathscr{P}(\mathcal{X}),\nu\in \mathscr{P}(\mathcal{Y})\), and \(c:\mathcal{X}\times \mathcal{Y}\to \mathbb{R}\) be a Borel measurable function.
The optimal transport problem associated to \(c\) is given by
\begin{equation*}
	\ot(\mu,\nu):=\inf_{\pi\in \Pi(\mu,\nu)}\E_{\pi}[c(X,Y)].
\end{equation*}
\begin{defn}
	Let \(f:\mathcal{X}\to\mathbb{R}\).
	We say \(f\) is \(c\)-convex if there exists \(g:\mathcal{Y}\to \mathbb{R}\)  such that
	\begin{equation*}
		f(x)=-\inf_{y\in \mathcal{Y}}\{c(x,y)+g(y)\}.
	\end{equation*}
	The \(c\)-supdifferential of \(f\) is defined as
	\begin{equation*}
		\partial^{c}f(x):=\{y \in \mathcal{Y}: f(x')\leq f(x) + c(x',y)-c(x,y) \text{ for any } x'\in \mathcal{X}\}.
	\end{equation*}
	We denote \(\{(x,y)\in \mathcal{X}\times \mathcal{Y}: y\in \partial^{c}f(x)\}\) by \(\partial^{c}f\).
\end{defn}
We recall the characterization of the optimal coupling from \citet[Theorem 2.7]{gangbo1996geometry}.
\begin{thm}
	\label{thm-ot}
	Let \(\pi\in \Pi(\mu,\nu)\).
	Then \(\pi\) is an optimal coupling if and only if \(\pi\) is supported on \(\partial^{c}f\) for some \(c\)-concave \(f\).
\end{thm}

Let \(\gamma\) be a reference probability measure on \( \mathscr{P}(\mathcal{X}\times \mathcal{Y})\).
The Schrödinger bridge is a coupling \(\pi\in \Pi(\mu,\nu)\) minimizing the relative entropy \(H(\pi|\gamma)\).
We first state an elementary inequality for the relative entropy.
\begin{lem}
	\label{lem-entropy}
	Let \(\pi\in \Pi(\mu,\nu)\) and \((f,g)\in L^{0}(\mathcal{X})\times L^{0}(\mathcal{Y})\).
	We have
	\begin{equation*}
		H(\pi|\gamma)\geq \E_{\mu}[f]+\E_{\nu}[g]-\E_{\gamma}[\exp(f\oplus g)] +1.
	\end{equation*}
\end{lem}
\begin{proof}
	For \(f\in L^{1}(\mu)\) and \(g\in L^{1}(\nu)\), we notice the above inequality holds.
	We recall that we follow the convention that \(\infty-\infty=-\infty\).
	It then suffices to show the inequality for  \(\E_{\mu}[f^{-}]<\infty\), \(\E_{\nu}[g^{-}]<\infty\), and \(\E_{\gamma}[\exp(f\oplus g)]<\infty\).
	We take \(f^{n}=f\wedge n \) and \(g^{n}=g\wedge n\).
	The desired inequality holds for \((f^{n},g^{n})\) as \(f^{n}\in L^{1}(\mu)\) and \(g^{n}\in L^{1}(\nu)\).
	By taking \(n\) to infinite, we conclude the proof from the monotone convergence theorem.
\end{proof}

We recall the duality result for the Schrödinger bridge problem from \citet[Theorem 3.2]{nutz2021introduction}.
\begin{thm}
	We assume that there exists \(f^{*}\in L^{1}(\mu)\), \(g^{*}\in L^{1}(\nu)\) and \(\pi^{*}\in \Pi(\mu,\nu)\) satisfying
	\begin{equation*}
		\frac{\md \pi^{*}}{\md \gamma}= \exp(f^{*}\oplus g^{*}).
	\end{equation*}
	It holds that
	\begin{equation*}
		\inf_{\pi\in \Pi(\mu,\nu)}H(\pi|\gamma)= \sup_{f\in L^{1}(\mu), g\in L^{1}(\nu)}\{\E_{\mu}[f]+\E_{\nu}[g]-\E_{\gamma}[\exp(f\oplus g)]+1\}.
	\end{equation*}
	In particular, \(\pi^{*}\) is the primal optimizer and \((f^{*},g^{*})\) is the dual optimizer.
\end{thm}

As the dual problem is concave, \((f^{*},g^{*})\) satisfies the first order optimality condition
\begin{equation*}
	\left\{
	\begin{aligned}
		f^{*}(x) & = -\log\Bigl(\int_{\mathcal{Y}} \exp(g^{*}(y))\frac{\md \gamma}{\md \mu\otimes \nu}(x,y)\nu(\md y)\Bigr), \\
		g^{*}(y) & = -\log\Bigl(\int_{\mathcal{X}} \exp(f^{*}(x))\frac{\md \gamma}{\md \mu\otimes \nu}(x,y)\mu(\md x)\Bigr).
	\end{aligned}
	\right.
\end{equation*}
It is also known as the Schrödinger equation system.
We recall the following well-posedness of the Schrödinger equation in a discrete setup.
\begin{thm}
	\label{thm-se}
	Let \(x_{i}\in \mathcal{X}\) and \(y_{j}\in \mathcal{Y}\) for \(1\leq i\leq n\) and \(1\leq j\leq m\).
	We consider \(\mu=\sum_{i=1}^{n}a_{i}\delta_{x_{i}}\), \(\nu=\sum_{j=1}^{m}b_{j}\delta_{y_{j}}\), and \(\sigma=\sum_{i=1}^{n}\sum_{j=1}^{m} \sigma_{ij} \delta_{(x_{i},y_{j})}\) with \(a_{i}\), \(b_{j}\), \(\sigma_{ij}\) positive and \(\sum_{i=1}^{n}a_{i}=\sum_{j=1}^{m}b_{j}=1\).
	The Schrödinger equation for \(\inf_{\pi\in\Pi(\mu,\nu)}H(\pi|\sigma)\) reads as
	\begin{equation}
		\label{eqn-discse}
		\left\{
		\begin{aligned}
			a_{i} & =\sum_{j=1}^{m}\sigma_{ij}\exp(p^{*}_{i}+q^{*}_{j}), \\
			b_{j} & =\sum_{i=1}^{n}\sigma_{ij}\exp(p^{*}_{i}+q^{*}_{j}), \\
		\end{aligned}
		\right.
	\end{equation}
	and its admits a unique solution \((p^{*},q^{*})\) up to the equivalence given by \(\sim_{\oplus}\).
	In particular,
	\begin{equation*}
		\argmin_{\pi\in\Pi(\mu,\nu)}H(\pi|\gamma)=\pi^{*}=\sum_{i=1}^{n}\sum_{j=1}^{m}\sigma_{ij}\exp(p^{*}_{i}+q^{*}_{j})\delta_{(x_{i},y_{j})}.
	\end{equation*}
	We say \((p^{*},q^{*})\) is the dual potential associated to the above Schrödinger bridge problem.
\end{thm}
\begin{rmk}
	\label{rmk-stab}
	The reference measure \(\gamma\) is not restricted to  probability measures.
	The existence and uniqueness  of the solution to \eqref{eqn-discse} follow from a simple rescaling.
	Moreover, by combining this uniqueness with the compactness of \(\Pi(\mu,\nu)\) we can show the Schrödinger equation is stable with respect to \(\sigma\).
	Namely, if  \(\lim_{n\to\infty}\sigma_{ij}^{n}=\sigma_{ij}^{\infty}>0\), then the corresponding Schrödinger bridge \(\pi^{*,n}\) converges weakly to the solution of the limit problem, and the dual potentials \((p^{*,n},q^{*,n})\) converges to the limit dual potentials in \(l^{\infty}_{\oplus}\).
\end{rmk}

\subsection{Measurable selection}
For a Polish space \(\mathcal{X}\), the universal \(\sigma\)-algebra on it is defined as \(\mathcal{U}(\mathcal{X})=\cap_{\mu\in \mathscr{P}(\mathcal{X})} \prescript{\mu}{}{}\mathcal{B}(\mathcal{X})\), where \(\prescript{\mu}{}{}\mathcal{B}(\mathcal{X})\) is the completion of \(\mathcal{B}(\mathcal{X})\) under \(\mu\).
In particular, for any universally measurable function \(\varphi\) and \(\mu\in \mathscr{P}(\mathcal{X})\), there exists \(f\in L^{0}(\mathcal{X})\) such that \(\varphi=f\) \(\mu\)--a.s.
We hence define
\begin{equation*}
	\E_{\mu}[\varphi]=\E_{\mu}[f].
\end{equation*}

We also recall the following proposition from \citet[Proposition 7.50]{bertsekas1996stochastic}.
\begin{prop}[Analytic selection theorem]
	\label{prop-sel}
	Let \(\mathcal{X}\) and \(\mathcal{Y}\) be Polish spaces and \(\varphi:\mathcal{X}\times \mathcal{Y}\to \mathbb{R}\) a Borel measurable function.
	Define \(\tilde{\varphi}:\mathcal{X}\to \mathbb{R}\cup \{-\infty\}\) by
	\begin{equation*}
		\tilde{\varphi}(x)=\inf_{y\in \mathcal{Y}}\varphi(x,y).
	\end{equation*}
	Then for any  \(\varepsilon>0\), there exists an analytically measurable function \(s:\mathcal{X}\to \mathcal{X}\) such that \((x,s(x))\in \mathcal{X}\times \mathcal{Y}\) for any \(x\in \mathcal{X}\), and
	\begin{equation*}
		\varphi(x,s(x))\leq\begin{cases}
			\tilde{\varphi}(x)+\varepsilon & \text{if}\quad \tilde{\varphi}(x)>-\infty,  \\
			-1/\varepsilon \quad           & \text{if} \quad \tilde{\varphi}(x)=-\infty.
		\end{cases}
	\end{equation*}
\end{prop}

\section{Duality, uniqueness, and existence}
\label{sec-tsb}
We recall the transport-relaxed Schrödinger bridge problem.
\begin{prob}
	\label{prob-s}
	Let \(c_{\mathcal{X}}:\mathcal{X}\times \mathcal{X}\to \mathbb{R}\) and \(c_{\mathcal{Y}}:\mathcal{Y}\times \mathcal{Y}\to \mathbb{R}\) be two Borel measurable functions.
	For any \(\mu\in \mathscr{P}(\mathcal{X})\), \(\nu\in \mathscr{P}(\mathcal{Y})\), and \(\gamma\in \mathscr{P}(\mathcal{X}\times \mathcal{Y})\), the transport-relaxed Schrödinger bridge is given by
	\begin{equation}
		\label{eqn-primal}
		\inf_{\pi\in \mathscr{P}(\mathcal{X}\times \mathcal{Y})}I(\pi):=\inf_{\pi\in \mathscr{P}(\mathcal{X}\times \mathcal{Y})}  \{\ot(\mu,\pi_{x})+\ot(\nu,\pi_{y}) +H(\pi | \gamma)\},
	\end{equation}
	where \(\ot\) denotes the optimal transport associated to \(c_{\mathcal{X}}\) and \(c_{\mathcal{Y}}\).
\end{prob}

\begin{rmk}
	By taking \(c_{\mathcal{X}}(x,x')=+\infty \1_{\{x\neq x'\}}\) and \(c_{\mathcal{Y}}(y,y')=+\infty \1_{\{y\neq y'\}}\),  we retrieve the classical Schrödinger bridge problem.
	By taking \(c_{\mathcal{X}}(x,x')=+\infty \1_{\{x\neq x'\}}\) and replacing \(\ot(\cdot,\nu)\)  with \(H(\cdot|\nu)\), we recover the relaxed Schrödinger bridge problem studied in \citet{garg24soft}.
\end{rmk}
\begin{rmk}
	If the reference measure \(\gamma(\md x, \md y)=\gamma_{x}(\md x)\otimes\gamma_{y}(\md y)\), then the relaxed Schrödinger bridge is reduced to
	\begin{equation*}
		\pi^{*} = \arginf_{\pi\in \Pi(\mu',\nu')}H(\pi |\gamma),
	\end{equation*}
	where \(\mu'=\arginf\{\ot(\mu,\cdot)+H(\cdot|\gamma_{x})\}\) and \(\nu'=\arginf\{\ot(\nu,\cdot)+H(\cdot|\gamma_{y})\}\).
\end{rmk}

\begin{asmp}
	\label{asmp-dual}
	We assume that there exist \(c\)-convex \(f^{*}\in L^{0}(\mathcal{X})\), \(g^{*}\in L^{0}(\mathcal{Y})\), and \(\pi^{*}\in \Pi(\mu,\nu)\) such that
	\begin{enumerate}[(i)]
		\item \(\md \pi^{*}= \exp(f^{*}\oplus g^{*})\md \gamma\);
		\item there exists \(\eta\in \Pi( \pi^{*}_{x},\mu)\) such that \(\supp(\eta)\subseteq \partial^{c} f^{*}\);
		\item there exists \(\lambda\in \Pi( \pi^{*}_{y},\nu)\) such that \(\supp(\lambda)\subseteq \partial^{c} g^{*}\).
	\end{enumerate}
\end{asmp}

We state our duality result.

\begin{thm}[Duality]
	\label{thm-dual}
	We write the dual problem as
	\begin{equation}
		\label{eqn-dual}
		\sup_{f\in L^{0}(\mathcal{X}),g\in L^{0}(\mathcal{Y})}J(f,g):=\sup_{f\in L^{0}(\mathcal{X}),g\in L^{0}(\mathcal{Y})} \{E_{\mu}[f^{c}] +\E_{\nu}[g^{c}] -\E_{\gamma}[\exp(f\oplus g)]+1\},
	\end{equation}
	where \(f^{c}(x)=\inf_{x'\in \mathcal{X}}\{f(x')+c_{\mathcal{X}}(x,x')\}\) and \(g^{c}(y)=\inf_{y'\in \mathcal{Y}}\{f(y')+c_{\mathcal{Y}}(y,y')\}\).

	Under Assumption \ref{asmp-dual}, the strong duality holds, i.e., \(\inf_{\pi}I(\pi)=\sup_{f,g}J(f,g)\).
	Moreover, \(\pi^{*}\) and \((f^{*},g^{*})\) in Assumption \ref{asmp-dual} are the primal and dual optimizers respectively.
\end{thm}

\begin{proof}
	Applying Lemma \ref{lem-entropy} to \(H(\pi|\gamma)\), we derive
	\begin{align*}
		I & \geq \inf_{(\mu',\nu')\in \mathscr{P}(\mathcal{X}) \times\mathscr{P}(\mathcal{Y})}\sup_{f,g\in L^{0}}\{(\ot(\mu,\mu')+\E_{\mu'}[f])+(\ot(\nu,\nu')+\E_{\nu'}[g])-\E_{\gamma}[\exp(f\oplus g)]\}+1   \\
		  & \geq \sup_{f,g\in L^{0}} \inf_{(\mu',\nu')\in \mathscr{P}(\mathcal{X}) \times\mathscr{P}(\mathcal{Y})}\{(\ot(\mu,\mu')+\E_{\mu'}[f])+(\ot(\nu,\nu')+\E_{\nu'}[g])-\E_{\gamma}[\exp(f\oplus g)]\}+1.
	\end{align*}
	To obtain the weak duality
	\(
	I \geq J
	\),
	it suffices to show that
	\begin{equation*}
		\inf_{\mu'\in \mathscr{P}(\mathcal{X})}\{\ot(\mu,\mu')+\E_{\mu'}[f]\}=\E_{\mu}[f^{c}] \text{ and } \inf_{\nu'\in \mathscr{P}(\mathcal{Y})}\{\ot(\nu,\nu')+\E_{\nu'}[g]\}=\E_{\mu}[g^{c}].
	\end{equation*}
	Let \(\varphi(x,x')=f(x')+c_{\mathcal{X}}(x,x')\).	By Proposition \ref{prop-sel}, for any \(\varepsilon\) there exists \(s:\mathcal{X}\to \mathcal{X}\) such that
	\begin{equation*}
		\varphi(x,s(x))\leq\begin{cases}
			f^{c}(x)+\varepsilon & \text{if}\quad f^{c}(x)>-\infty,  \\
			-1/\varepsilon \quad & \text{if} \quad f^{c}(x)=-\infty.
		\end{cases}
	\end{equation*}
	By taking \(\hat{\mu}=s_{\#}\mu\), we have
	\begin{equation*}
		\inf_{\mu'\in \mathscr{P}(\mathcal{X})}\{\ot(\mu,\mu')+\E_{\mu'}[f]\}\leq \ot(\mu,\hat{\mu})+\E_{\hat{\mu}}[f]=\E_{\mu}[\phi(X,s(X))]\leq \E_{\mu}\Bigl[\min\Bigl\{f^{c},-\frac{1}{\varepsilon}\Bigr\}\Bigr]+\varepsilon.
	\end{equation*}
	As \(\varepsilon\) goes to 0, we obtain
	\begin{equation*}
		\inf_{\mu'\in \mathscr{P}(\mathcal{X})}\{\ot(\mu,\mu')+\E_{\mu'}[f]\}\leq \E_{\mu}[f^{c}].
	\end{equation*}
	For the reverse direction, we write
	\begin{equation*}
		\inf_{\mu'\in \mathscr{P}(\mathcal{X})}\{\ot(\mu,\mu')+\E_{\mu'}[f]\}=\inf_{\Pi(\mu,*)}\E_{\pi}[f(X')+c_{\mathcal{X}}(X,X')]\geq \inf_{\Pi(\mu,*)}[f^{c}(X)]=\E_{\mu}[f^{c}].\end{equation*}
	Similarly, we obtain \(\inf_{\nu'\in \mathscr{P}(\mathcal{Y})}\{\ot(\nu,\nu')+\E_{\nu'}[g]\}=\E_{\mu}[g^{c}]\).

	Let \(\pi^{*}\) and \((f^{*},g^{*})\) be as in Assumption \ref{asmp-dual}.
	On the other hand, by Assumption \ref{asmp-dual}(ii) and Theorem \ref{thm-ot} we have
	\begin{equation*}
		\ot(\mu,\pi^{*}_{x})= \E_{\eta}[c_{\mathcal{X}}(X,X')] =\E_{\mu}[(f^{*})^{c}]-\E_{\pi^{*}_{x}}[f^{*}].
	\end{equation*}
	Similarly, we have
	\begin{equation*}
		\ot(\nu,\pi^{*}_{y})=\E_{\nu}[(g^{*})^{c}]-\E_{\pi^{*}_{y}}[g^{*}].
	\end{equation*}
	Hence, we deduce
	\begin{align*}
		J & \geq E_{\mu}[(f^{*})^{c}] +\E_{\nu}[(g^{*})^{c}]                                               \\
		  & = \ot(\mu, \pi^{*}_{x}) +\ot(\nu,\pi^{*}_{y}) +\E_{\pi^{*}_{x}}[f^{*}]+\E_{\pi^{*}_{y}}[g^{*}] \\
		  & =  \ot(\mu, \pi^{*}_{x}) +\ot(\nu,\pi^{*}_{y}) + H(\pi^{*}|\gamma) \geq I.
	\end{align*}
	Moreover, we notice that the primal and dual optimizers are given by \(\pi^{*}\) and \((f^{*},g^{*})\).
\end{proof}

We discuss the uniqueness and existence for the primal problem.
\begin{prop}
	\label{prop-primal}
	We have	\(\pi\mapsto I(\pi)\) in the primal problem \eqref{eqn-primal} is strictly convex.
	If we further assume \(c_{\mathcal{X}}\) and \(c_{\mathcal{Y}}\) are lower semi-continuous and bounded from below then \(\pi\mapsto I(\pi)\) is lower semi-continuous w.r.t the total variation topology.
	In particular, if there exists \(\pi\) such that \(I(\pi)<\infty\), then the primal problem exists a unique optimizer.
\end{prop}
\begin{proof}
	We first notice that \(\ot(\mu,\cdot)\), \(\ot(\nu,\cdot)\) are convex and \(H(\cdot|\gamma)\) is strictly convex.
	This leads to the strict convexity of \(I\).
	We note the variation formulation \citep[Lemma 1.3]{nutz2021introduction}
	\begin{equation*}
		H(\pi|\gamma)=\sup_{\varphi \in L^{0}(\mathcal{X}\times \mathcal{Y}), \text{ bounded}}\{\E_{\pi}[\varphi]-\log(\E_{\gamma}[\exp(\varphi)])\},
	\end{equation*}
	which yields the lower semi-continuity of \(H(\cdot|\gamma)\) with respect to the total variation topology.
	Moreover, let \(\mu_{n}'\to \mu'\) be a convergent sequence in total variation.
	For any sequence \(\pi_{n}\in \Pi(\mu,\mu_{n}')\) there exists \(\pi\in \Pi(\mu,\mu')\) such that \(\pi_{n}\) converges to \(\pi\) in weak topology.
	As \(c_{\mathcal{X}}\) is lower semi-continuous and bounded from below, by Portmanteau theorem we obtain
	\begin{equation*}
		\E_{\pi}[c]\leq \liminf_{n\to\infty}\E_{\pi_{n}}[c].
	\end{equation*}
	This implies \(\ot(\mu,\cdot)\) is lower semi-continuous, and the same holds for \(\ot(\nu,\cdot)\).
	Therefore, we deduce \(I\) is lower semi-continuous under the total variation topology.

	If now there exists \(\pi\) such that \(I(\pi)<\infty\), then we take a minimizing sequence \(\{\pi_{n}\}_{n\geq 1}\).
	In particular, we can take \(\{\pi_{n}\}_{n\geq 1}\) such that \(H(\pi_{n}|\gamma)\leq I(\pi)- \inf c_{\mathcal{X}}- \inf c_{\mathcal{Y}}<\infty.\)
	By \citet[Lemma 1.8]{nutz2021introduction}, there exists \(\pi_{n}'\in \operatorname{conv}\{\pi_{n},\pi_{n+1},\dots\}\) such that \(\pi_{n}\) converges to a limit \(\pi^{*}\) in total variation.
	It follows, from the convexity and lower semi-continuity of \(I\), that
	\begin{equation*}
		I(\pi^{*})\leq \liminf_{n\to\infty} I(\pi_{n}')\leq \liminf_{n\to\infty} \sup_{m\geq n}I(\pi_{m})=\lim_{n\to \infty}I(\pi_{n}).
	\end{equation*}
	Hence, \(\pi^{*}\) is a primal optimizer.
	The uniqueness follows directly from the strict convexity.
\end{proof}

In classical optimal transport, the uniqueness and existence of optimal potentials is in general open and greatly depends on the geometry of the underlying space and the choice of the cost function.
Henceforth, we focus on the Euclidean space \(\mathcal{X}=\mathcal{Y}=\mathbb{R}^{d}\) and the quadratic cost \(c_{\mathcal{X}}=c_{\mathcal{Y}}=\frac{1}{2}\|\cdot-\cdot\|^{2}\).
Our aim is to establish the uniqueness and existence of potentials for the relaxed Schrödinger bridge.

Motivated by data-driven applications, we consider the following semi-discrete setup.
\begin{asmp}
	\label{asmp-semi}
	Let \(\mu=\sum_{i=1}^{n}a_{i}\delta_{x_{i}}\) and \(\nu=\sum_{j=1}^{m}b_{j}\delta_{y_{j}}\) be discrete probability measures with \(\sum_{i=1}^{n}a_{i}=\sum_{j=1}^{m}b_{j}=1\), and \(\gamma\) is absolutely continuous with respect to the Lebesgue  measure.
\end{asmp}

For \(\alpha\in \mathbb{R}^{n}\) and \(\beta\in \mathbb{R}^{m}\), we write \(f(x,\alpha)=\sup_{1\leq i \leq n}\{\langle x, x_{i} \rangle+\alpha_{i}\}\) and \(g(y,\beta)=\sup_{1\leq j \leq m}\{\langle y, y_{j} \rangle+\beta_{j}\}\).
We also write
\begin{equation*}
	\tilde{\gamma}(\md x,\md y)= \exp\Bigl(-\frac{1}{2}\|x\|^{2}-\frac{1}{2}\|y\|^{2}\Bigr) \gamma(\md x, \md y).
\end{equation*}
In the next theorem, we show in a semi-discrete setting the dual problem \eqref{eqn-dual} is equivalent to a \emph{finite-dimensional} concave optimization problem and the strong duality holds.
\begin{thm}
	\label{thm-semi}
	Let Assumption \ref{asmp-semi} hold.
	Recall that we write
	\begin{equation*}
		U(\alpha,\beta) :=\sum_{i=1}^{n}a_{i}(\alpha_{i}+\frac{1}{2}\|x_{i}\|^{2})+\sum_{j=1}^{m}b_{j}(\beta_{j}+\frac{1}{2}\|y_{j}\|^{2})-\int\exp(f(x,\alpha)+g(y,\beta))\md \tilde{\gamma}+1.
	\end{equation*}
	Then we have the strong duality:
	\begin{equation*}
		\sup_{(\alpha,\beta)\in \mathbb{R}^{n+m}}U(\alpha,\beta)=\sup_{f,g\in L^{0}(\mathbb{R}^{d})}J(f,g)=\inf_{\pi\in \mathscr{P}(\mathbb{R}^{d}\times \mathbb{R}^{d})}I(\pi).
	\end{equation*}
	Moreover, the dual problem has a unique optimizer \((\alpha^{*},\beta^{*})\) in \(\mathbb{R}^{n+m}_{\oplus}\) and the primal optimizer is given by
	\begin{equation*}
		\md \pi^{*}=\exp(f(x,\alpha^{*})+g(y,\beta^{*})) \md \tilde{\gamma}.
	\end{equation*}
\end{thm}
\begin{rmk}
	We remark that  the dual optimizer \((\alpha^{*},\beta^{*})\) has a trivial multiplicity as for any \(r\in \mathbb{R}\), the pair \((\alpha,\beta)=(\alpha^{*},\beta^{*})+r(\mathbf{1},-\mathbf{1})\) is also an optimizer.
	Let \(U_{\oplus}:\mathbb{R}^{n+m}_{\oplus}\to \mathbb{R}\) be the map induced by \(U\) on the quotient space.
	We stress that even after quotienting out this trivial multiplicity,  \(U_{\oplus}\) is not a strictly concave function.
	In particular, the uniqueness of the dual optimizer in \(\mathbb{R}^{n+m}_{\oplus}\) does not follow from the strict concavity; instead, it is a consequence of the strict convexity of the relative entropy in the primal problem.
\end{rmk}
\begin{proof}
	\emph{Step 1.} We first establish the  existence of an optimizer for the dual problem.
	As \(U\) is continuous, it suffices to show that the projection of the upper level set
	\begin{equation*}
		\proj_{\oplus}(Z_{K}):=\proj_{\oplus}(\{(\alpha,\beta):U(\alpha,\beta)\geq K \})\subseteq \mathbb{R}^{n+m}_{\oplus}
	\end{equation*}
	is a non-empty compact set for sufficiently large \(K\).
	We note that
	\begin{equation*}
		U(\alpha,\beta)  \leq 1+ \sum_{i=1}^{n}\frac{1}{2}a_{i}\|x_{i}\|^{2}+\sum_{j=1}^{m}\frac{1}{2}b_{j}\|y_{j}\|^{2}+\overline{\alpha}+\overline{\beta}-\exp(\overline{\alpha}+\overline{\beta})\int\exp(\inf_{1\leq i\leq n}\{\langle x, x_{i} \rangle\}+\inf_{1\leq j\leq m}\{\langle y, y_{j} \rangle\})\md \tilde{\gamma}.	\end{equation*}
	Hence, we deduce \(\sup_{(\alpha,\beta)\in Z_{K}} (\overline{\alpha}+\overline{\beta})<\infty\).
	Moreover,  the estimate
	\begin{align*}
		U(\alpha,\beta) & \leq 1+ \sum_{i=1}^{n}\frac{1}{2}a_{i}\|x_{i}\|^{2}+\sum_{j=1}^{m}\frac{1}{2}b_{j}\|y_{j}\|^{2}+\sum_{i=1}a_{i}\alpha_{i}+\sum_{j=1}b_{j}\beta_{j}                                                                         \\
		                & \leq 1+ \sum_{i=1}^{n}\frac{1}{2}a_{i}\|x_{i}\|^{2}+\sum_{j=1}^{m}\frac{1}{2}b_{j}\|y_{j}\|^{2}+ \underline{\alpha}\min_{1\leq i\leq n}a_{i}+\underline{\beta}\min_{1\leq j\leq m}b_{j}+\overline{\alpha}+\overline{\beta}
	\end{align*}
	yields that \(\inf_{(\alpha,\beta)\in Z_{K}} (\underline{\alpha}+\underline{\beta})>-\infty\).
	Therefore, we obtain for sufficiently large \(R\)
	\begin{equation*}
		Z_{K}\subseteq \{(\alpha,\beta):-R\leq \underline{\alpha}+\underline{\beta}\leq  \overline{\alpha}+\overline{\beta}\leq R\},
	\end{equation*}
	from which the compactness of \(\proj_{\oplus}Z_{K}\) follows immediately.

	\emph{Step 2.}
	Let \((\alpha^{*},\beta^{*})\) be an optimizer.
	We define
	\begin{equation*}
		f^{*}(x)=f(x,\alpha^{*})-\frac{1}{2}\|x\|^{2},\; g^{*}(y)=g(y,\beta)-\frac{1}{2}\|y\|^{2}, \text{ and } \md \pi^{*}= \exp(f^{*}\oplus g^{*})\md \gamma.
	\end{equation*}
	We verify \((f^{*},g^{*})\) and \(\pi^{*}\) satisfy all conditions in Assumption \ref{asmp-dual}.
	Note that \(f(x,\alpha)+g(y,\beta)\) is  convex in \((\alpha,\beta)\) for any fixed \((x,y)\) which implies \(U\) is a concave function of \((\alpha,\beta)\).
	As \(\alpha \mapsto f(X,\alpha)\) is differentiable \(\gamma\)-a.s., from the first order optimality condition of \((\alpha^{*},\beta^{*})\) we deduce that for \(1\leq i\leq n \) and \(1\leq j\leq m\)
	\begin{equation}
		\label{eqn-se}
		\left\{
		\begin{aligned}
			\partial_{\alpha_{i}}U(\alpha^{*},\beta^{*}) & =a_{i}-\E_{\gamma}\bigl[\exp(f^{*}(X)+g^{*}(Y))\1_{A_{i}(\alpha^{*})}(X)\bigr]=0, \\
			\partial_{\beta_{j}}U(\alpha^{*},\beta^{*})  & =b_{j}-\E_{\gamma}\bigl[\exp(f^{*}(X)+g^{*}(Y))\1_{B_{i}(\beta^{*})}(Y)\bigr]=0,
		\end{aligned}
		\right.
	\end{equation}
	where \(A_{i}\) and \(B_{j}\) are Laguerre cells given by
	\begin{equation*}
		A_i(\alpha) :=\{x:f(x,\alpha)=\langle x,x_i\rangle +\alpha_i\} \text{ and } B_j(\beta) :=\{y:g(y,\beta)=\langle y,y_j\rangle +\beta_j\}.
	\end{equation*}
	As \(\sum_{i=1}^{n}a_{i}=\sum_{j=1}^{m}b_{j}=1\), we sum up equations \eqref{eqn-se} and obtain \(\E_{\gamma}[\exp(f^{*}\oplus g^{*})]=1\), and hence, \(\pi^{*}\) is a probability measure.
	Moreover, \(f^{*}\) and \(g^{*}\) are \(c\)-convex since \(f(x,\alpha^{*})=f^{*}(x)+\frac{1}{2}\|x\|^{2}\) and \(g(y,\beta^{*})=g^{*}(y)+\frac{1}{2}\|y\|^{2}\) are convex.
	It is clear from the definition that
	\begin{equation*}
		\partial^{c}f^{*}(x)=x_{i} \text{ for } x\in A_{i}(\alpha^{*})  \text{ and } \partial^{c}g^{*}(y)=y_{j} \text{ for } y\in B_{j}(\beta^{*}).
	\end{equation*}
	Together with \eqref{eqn-se}, we verify
	\begin{equation*}
		\eta=(\Id,\partial^{c}f^{*})_{\#} \pi^{*}_{x}\in \Pi(\pi^{*}_{x},\mu) \text{ and } \lambda=(\Id,\partial^{c}g^{*})_{\#} \pi^{*}_{y}\in \Pi(\pi^{*}_{y},\nu).
	\end{equation*}
	Therefore, by Theorem \ref{thm-dual} we deduce
	\begin{align*}
		\inf_{\pi\in \mathscr{P}(\mathbb{R}^{d}\times \mathbb{R}^{d})}I(\pi) & =\sup_{f,g\in L^{0}(\mathbb{R}^{d})}J(f,g)                                                   \\
		                                                                     & =J(f^{*},g^{*})=\sum_{i=1}^{n}a_{i}(f^{*})^{c}(x_{i})+\sum_{j=1}^{m}b_{j}(g^{*})^{c}(y_{j}).
	\end{align*}
	As  \(c_{\mathcal{X}}(x,x')=\frac{1}{2}\|x-x'\|^{2}\) and \(x_{i}\in \partial^{c}f^{*}(x')\) for \(x'\in A_{i}(\alpha^{*})\), we have
	\begin{equation*}
		(f^{*})^{c}(x_{i})=\frac{1}{2}\|x-x'\|^{2} + f^{*}(x')= \alpha_{i}+ \frac{1}{2}\|x_{i}\|^{2} .
	\end{equation*}
	Similarly, we have
	\((g^{*})^{c}(y_{j})= \beta_{j}+ \frac{1}{2}\|y_{j}\|^{2}.\)
	We complete the duality by noticing
	\begin{equation*}
		\sup_{(\alpha,\beta)\in \mathbb{R}^{n+m}}U(\alpha,\beta)=U(\alpha^{*},\beta^{*})=\sup_{f,g\in L^{0}(\mathbb{R}^{d})}J(f,g)=\inf_{\pi\in \mathscr{P}(\mathbb{R}^{d}\times \mathbb{R}^{d})}I(\pi).
	\end{equation*}

	\emph{Step 3.}
	We show that the dual problem has a unique optimizer in \(\mathbb{R}^{n+m}_{\oplus}\).
	By Proposition~\ref{prop-primal}, \(\pi^{*}\) constructed in the previous step is the unique primal optimizer.
	Therefore, for any dual optimizer \((\alpha,\beta)\) we must have
	\begin{equation}
		\label{eqn-density}
		f(X,\alpha)+g(Y,\beta)=f(X,\alpha^{*})+g(Y,\beta^{*}) \quad \gamma\text{-a.s.}
	\end{equation}
	From the first order conditions \eqref{eqn-se}, we notice \(A_{i}(\alpha^{*})\) and \(B_{j}(\beta^{*})\) have non-empty interiors and occupy positive measures under \(\gamma\).
	Taking derivatives on both sides of \eqref{eqn-density} yields \(A_{i}(\alpha)=A_{i}(\alpha^{*})\) and \(B_{j}(\beta)=B_{j}(\beta^{*})\).
	Moreover, we have \(\alpha_{i}+\beta_{j}=\alpha^{*}_{i}+\beta^{*}_{j}\) for any \(1\leq i\leq n\) and \(1\leq j\leq m\), which implies \((\alpha,\beta)\sim_{\oplus}(\alpha^{*},\beta^{*})\).
\end{proof}

\section{Blow-up penalty  limit}
\label{sec-blowup}
In this section, we discuss the limit of the relaxed Schrödinger bridges as the transport penalty blows up.
We recall that
\begin{equation}
	\label{eqn-blowup}
	\inf_{\pi\in \mathscr{P}(\mathcal{X}\times \mathcal{Y})}I^{\varepsilon}(\pi):=\inf_{\pi\in \mathscr{P}(\mathcal{X}\times \mathcal{Y})}  \{\varepsilon^{-1}\ot(\mu,\pi_{x})+ \varepsilon^{-1}\ot(\nu,\pi_{y}) +H(\pi | \gamma)\},
\end{equation}
and  its optimizer is denoted by \(\pi^{*,\varepsilon}\).
We first start with a simpler case where the penalty does not blow up.
In this case, the transport-relaxed Schrödinger bridge converges to the classical Schrödinger bridge.
\begin{prop}
	\label{prop-blowup}
	Assume that \(\{\pi\in \Pi(\mu,\nu): H(\pi|\gamma)<\infty\}\) is non-empty.
	We have \(\pi^{*,\varepsilon}\) converges to \(\pi^{*,0}\) in the weak topology as \(\varepsilon\) goes to 0, where \(\pi^{*,0}\) is the unique solution to
	\begin{equation*}
		\inf_{\pi\in \Pi(\mu,\nu)}H(\pi|\gamma).
	\end{equation*}
\end{prop}
\begin{proof}
	The uniqueness and existence of the \(\pi^{*,0}\) follows directly from the strict convexity and the lower-semi continuity of the relative entropy \(H\).
	We notice that  \(\inf_{\pi\in \mathscr{P}(\mathcal{X}\times \mathcal{Y})}I^{\varepsilon}(\pi)\leq I^{\varepsilon}(\pi^{*,0})\) is uniformly bounded.
	This yields the convergence of \(\pi_{x}^{*,\varepsilon}\) to \(\pi_{x}^{*,0}\) and the convergence of \(\pi_{y}^{*,\varepsilon}\) to \(\pi_{y}^{*,0}\) in the weak topology.
	Therefore, \(\{\pi^{*,\varepsilon}\}_{0<\varepsilon<1}\) forms a precompact set.
	Let \(\varepsilon_{n}\to 0\) and \(\{\pi^{*,\varepsilon_{n}}\}\) be a converging subsequence with limit \(\hat{\pi}\).
	From \(I^{\varepsilon}(\pi^{*,\varepsilon})\leq I^{\varepsilon}(\pi^{*,0})\), we derive
	\begin{equation*}
		H(\hat{\pi}|\gamma)\leq H(\pi^{*,0}|\gamma).
	\end{equation*}
	By the uniqueness of \(\pi^{*,0}\), we must have \(\tilde{\pi}=\pi^{*,0}\), and we deduce the convergence of \(\pi^{*,\varepsilon}\).
\end{proof}

It is obvious that under the semi-discrete setting (Assumption \ref{asmp-semi}),
\[
	\{\pi\in \Pi(\mu,\nu): H(\pi|\gamma)<\infty\}=\emptyset,
\]
and the classical Schrödinger bridge does not admit a solution.
Nevertheless, we will show that \(\pi^{*,\varepsilon}\) still converges to a limit coupling \(\pi^{*,0}\in \Pi(\mu,\nu)\).
Moreover, the limit \(\pi^{*,0}\) is the solution of a discrete Schrödinger bridge problem depending only on the local property of the reference measure \(\gamma\).

\begin{asmp}
	\label{asmp-blowup}
	We assume that \(\gamma\) has a positive and bounded density \(\rho\).
	Moreover, we assume \(\rho\) is locally continuous at \((x_{i},y_{j})\) for any \(1\leq i\leq n\) and \(1\leq j\leq m\).
\end{asmp}

\begin{thm}
	\label{thm-blowup}
	Let  Assumptions \ref{asmp-semi} and \ref{asmp-blowup} hold.
	We have \(\pi^{*,\varepsilon}\) converges to \(\pi^{*,0}\) in the weak topology as \(\varepsilon\) goes to 0, where \(\pi^{*,0}\) is the unique solution to
	\begin{equation}
		\label{eqn-disc}
		\inf_{\pi\in \Pi(\mu,\nu)}H(\pi | \sigma ).
	\end{equation}
	Here, \(\sigma\) is a discrete measure given by \(\sigma=\sum_{i=1}^{n}\sum_{j=1}^{m}(2 \pi)^{d}\rho(x_{i},y_{j})\delta_{(x_{i},y_{j})}\).
	Furthermore, we have the following expansion of \(I^{\varepsilon}(\pi)\):
	\begin{equation*}
		I^{\varepsilon}(\pi^{*,\varepsilon}) = -d \ln \varepsilon +H(\pi^{*,0}|\sigma) + o(1).
	\end{equation*}
\end{thm}
\begin{rmk}
	We remark that the leading-order blow-up of \(I^{\varepsilon}(\pi^{*,\varepsilon})\) is given by \(-d\ln \varepsilon\).
	This term is intrinsic --- only depends on  \(d\), the dimension of the underlying space and is independent of the choice of \(\gamma\), \(\mu\), \(\nu\).
\end{rmk}

\begin{proof}
	Similar to Theorem \ref{thm-semi}, the dual of \eqref{eqn-blowup} can be formulated as a finite-dimensional maximization \(\sup_{(\alpha,\beta)\in \mathbb{R}^{n+m}}U^{\varepsilon}(\alpha,\beta)\) where
	\begin{equation}
		\label{eqn-dual-blowup}
		\begin{aligned}
			\qquad U^{\varepsilon}(\alpha,\beta) & := \sum_{i=1}^{n}\varepsilon^{-1} a_{i}(\alpha_{i}+\frac{1}{2}\|x_{i}\|^{2})+\sum_{j=1}^{m}\varepsilon^{-1}b_{j}(\beta_{j}+\frac{1}{2}\|y_{j}\|^{2}) \\
			                                     & \quad -\int\exp(\varepsilon^{-1}(f(x,\alpha)+g(y,\beta)-\frac{1}{2}\|x\|^{2}-\frac{1}{2}\|y\|^{2}))\md \gamma +1.
		\end{aligned}
	\end{equation}
	We denote the dual optimizer by \((\alpha^{*,\varepsilon},\beta^{*,\varepsilon})\).

	\emph{Step 1.} We show \(\{(\alpha^{*,\varepsilon},\beta^{*,\varepsilon})\}_{0<\varepsilon<1}\) is precompact in \(\mathbb{R}^{n+m}_{\oplus}\) through a prior estimate.
	We note that
	\begin{align*}
		0 & \leq \varepsilon U^{\varepsilon}(\alpha^{*,\varepsilon},\beta^{*,\varepsilon})                                                                                                                                                                                                                                                                                     \\
		  & = \sum_{i=1}^{n} a_{i}(\alpha_{i}^{*,\varepsilon}+\frac{1}{2}\|x_{i}\|^{2})+\sum_{j=1}^{m}b_{j}(\beta_{j}^{*,\varepsilon}+\frac{1}{2}\|y_{j}\|^{2})                                                                                                                                                                                                                \\
		  & \quad -\varepsilon \int\exp\biggl(\varepsilon^{-1}\Bigl(f(x,\alpha^{*,\varepsilon})+g(y,\beta^{*,\varepsilon})-\frac{1}{2}\|x\|^{2}-\frac{1}{2}\|y\|^{2}\Bigr)\biggr)\md \gamma +1                                                                                                                                                                                 \\
		  & \leq \sum_{i=1}^{n} \frac{1}{2}a_{i}\|x_{i}\|^{2}+\sum_{j=1}^{m}\frac{1}{2}b_{j}\|y_{j}\|^{2} +1                                                                                                                                                                                                                                                                   \\
		  & \quad + \overline{\alpha}^{*,\varepsilon} +\overline{\beta}^{*,\varepsilon} - \varepsilon \int \exp\biggl(\varepsilon^{-1} \Bigl(\inf_{1\leq i\leq n}\langle x, x_{i} \rangle+\inf_{1\leq j\leq m}\langle y, y_{j} \rangle -\frac{1}{2}\|x\|^{2}-\frac{1}{2}\|y\|^{2}+ \overline{\alpha}^{*,\varepsilon} +\overline{\beta}^{*,\varepsilon}\Bigr)\biggr)\md \gamma.
	\end{align*}
	Assume on \(\{(x,y):\|x\|+\|y\| \leq \delta\}\) we have
	\begin{equation*}
		\inf_{1\leq i\leq n}\langle x, x_{i} \rangle+\inf_{1\leq j\leq m}\langle y, y_{j} \rangle -\frac{1}{2}\|x\|^{2}-\frac{1}{2}\|y\|^{2}\geq -M.
	\end{equation*}
	Then plugging it into the above estimate yields
	\begin{align*}
		0 & \leq \sum_{i=1}^{n} a_{i}\frac{1}{2}\|x_{i}\|^{2}+\sum_{j=1}^{m}b_{j}\frac{1}{2}\|y_{j}\|^{2} +1                                                                                                                              \\
		  & \quad + \overline{\alpha}^{*,\varepsilon} +\overline{\beta}^{*,\varepsilon}- \varepsilon \exp(\varepsilon^{-1}(-M +\overline{\alpha}^{*,\varepsilon} +\overline{\beta}^{*,\varepsilon})) \gamma(\{\|x\|+\|y\|\leq \delta \}).
	\end{align*}
	By Assumption \ref{asmp-blowup}, we have \(\gamma(\{\|x\|+\|y\|\leq \delta \})>0\) which further implies there exists \(R>0\) independent of \(\varepsilon\) such that
	\begin{equation*}
		\overline{\alpha}^{*,\varepsilon} +\overline{\beta}^{*,\varepsilon} \leq R.
	\end{equation*}
	Moreover, we notice
	\begin{align*}
		0 & \leq \varepsilon U^{\varepsilon}(\alpha^{*,\varepsilon},\beta^{*,\varepsilon})                                                                                                                                                                                                                \\
		  & \leq 1+ \sum_{i=1}^{n} a_{i}(\alpha_{i}^{*,\varepsilon}+\frac{1}{2}\|x_{i}\|^{2})+\sum_{j=1}^{m}b_{j}(\beta_{j}^{*,\varepsilon}+\frac{1}{2}\|y_{j}\|^{2})                                                                                                                                     \\
		  & \leq 1+ \sum_{i=1}^{n}\frac{1}{2}a_{i}\|x_{i}\|^{2}+\sum_{j=1}^{m}\frac{1}{2}b_{j}\|y_{j}\|^{2}+ \underline{\alpha}^{*,\varepsilon}\min_{1\leq i \leq n}a_{i}+\underline{\beta}^{*,\varepsilon}\min_{1\leq j \leq m}b_{j}+\overline{\alpha}^{*,\varepsilon}+\overline{\beta}^{*,\varepsilon}.
	\end{align*}
	Therefore, we deduce  \((\alpha^{*,\varepsilon},\beta^{*,\varepsilon})\) is contained in the compact set \(B_{R}:=\{(\alpha,\beta):-R\leq \underline{\alpha}+\underline{\beta}\leq  \overline{\alpha}+\overline{\beta}\leq R\}\) for sufficiently large \(R\).

	\emph{Step 2.} We claim that \((\alpha^{*,\varepsilon},\beta^{*,\varepsilon})\) converges to a limit \((\alpha^{*,0},\beta^{*,0})\) where \(\alpha^{*,0}_{i}=-\frac{1}{2}\|x_{i}\|^{2}\) and \(\beta^{*,0}_{j}=-\frac{1}{2}\|y_{j}\|^{2}\).
	It suffices to show that for any converging subsequence \(\{(\alpha^{*,\varepsilon_{n}},\beta^{*,\varepsilon_{n}})\}_{n\geq 1}\)  with \(\varepsilon_{n}\to 0\) its limit \((\hat{\alpha},\hat{\beta})\) coincides with \((\alpha^{*,0},\beta^{*,0})\).
	Assume there exist \((x_{0},y_{0})\) and \(r,\delta>0\) such that on \(\{(x,y):\|x-x_{0}\|+\|y-y_{0}\|\leq r\}\)
	\begin{equation}
		\label{eqn-bound}
		f(x,\hat{\alpha})+g(y,\hat{\beta})> \frac{1}{2}\|x\|^{2}+\frac{1}{2}\|y\|^{2}+\delta.
	\end{equation}
	We notice that
	\begin{align*}
		0 & \leq \limsup_{n\to \infty}\varepsilon_{n} U^{\varepsilon_{n}}(\alpha^{*,\varepsilon_{n}}, \beta^{*,\varepsilon_{n}})                                                                                                    \\
		  & = \sum_{i=1}^{n} a_{i}(\hat{\alpha}_{i}+\frac{1}{2}\|x_{i}\|^{2})+\sum_{j=1}^{m}b_{j}(\hat{\beta}_{j}+\frac{1}{2}\|y_{j}\|^{2})                                                                                         \\
		  & \quad -\liminf_{n\to\infty }\varepsilon_{n} \int\exp\biggl(\varepsilon_{n}^{-1}\Bigl(f(x,\alpha^{*,\varepsilon_{n}})+g(y,\beta^{*,\varepsilon_{n}})-\frac{1}{2}\|x\|^{2}-\frac{1}{2}\|y\|^{2}\Bigr)\biggr)\md \gamma +1 \\
		  & \leq \sum_{i=1}^{n} a_{i}(\hat{\alpha}_{i}+\frac{1}{2}\|x_{i}\|^{2})+\sum_{j=1}^{m}b_{j}(\hat{\beta}_{j}+\frac{1}{2}\|y_{j}\|^{2})                                                                                      \\
		  & \quad  - \liminf_{n\to\infty}\varepsilon_{n} \exp(\varepsilon_{n}^{-1}\delta) \gamma(\{\|x-x_{0}\|+\|y-y_{0}\|\leq r\}).
	\end{align*}
	Since we assume \(\gamma\) has a positive density, the upper bound above goes to \(-\infty\) as \(\varepsilon_{n}\to 0\).
	Therefore, we show that \eqref{eqn-bound} cannot be true, and hence \(f(x,\hat{\alpha})+g(y,\hat{\beta})\leq \frac{1}{2}\|x\|^{2}+ \frac{1}{2}\|y\|^{2}\).
	This yields that for any \(1\leq i\leq n\) and \(1\leq j\leq m\)
	\[
		\langle x, x_{i} \rangle + \langle y, y_{j} \rangle +\hat{\alpha}_{i}+ \hat{\beta}_{j}\leq \frac{1}{2}\|x\|^{2}+\frac{1}{2}\|y\|^{2},
	\]
	which implies \(\hat{\alpha}_{i}+\hat{\beta}_{j}\leq -\frac{1}{2}\|x_{i}\|^{2}-\frac{1}{2}\|y_{j}\|^{2}\).
	Now we assume
	\begin{equation}
		\label{eqn-lower}
		\hat{\alpha}_{i}+\hat{\beta}_{j}< -\frac{1}{2}\|x_{i}\|^{2}-\frac{1}{2}\|y_{j}\|^{2} \text{ for some } 1\leq i\leq n \text{ and } 1\leq j\leq m.
	\end{equation}
	Hence, we have either \(\hat{\alpha}_{i}+\hat{\beta}_{l}< -\frac{1}{2}\|x_{i}\|^{2}-\frac{1}{2}\|y_{l}\|^{2}\) for any \(1\leq l\leq m\) or  \(\hat{\alpha}_{k}+\hat{\beta}_{j}< -\frac{1}{2}\|x_{k}\|^{2}-\frac{1}{2}\|y_{j}\|^{2} \) for any \(1\leq k\leq m\).
	Without loss of generality, we assume the former case.
	In particular, for sufficiently large \(n\) we have
	\begin{equation*}
		f(x,\alpha^{*,\varepsilon_{n}})+g(y,\beta^{*,\varepsilon_{n}})-\frac{1}{2}\|x\|^{2}-\frac{1}{2}\|y\|^{2}<0  \text{ on } A_{i}(\alpha^{*,\varepsilon_{n}})\times \mathcal{Y}.
	\end{equation*}
	Since \((\alpha^{*,\varepsilon_{n}},\beta^{*,\varepsilon_{n}})\) is the optimizer of \eqref{eqn-dual-blowup}, it satisfies the first order condition:
	\begin{equation*}
		a_{i}  =  \int_{A_{i}(\alpha^{*,\varepsilon_{n}})\times \mathcal{Y}} \exp\biggl(\varepsilon_{n}^{-1}\Bigl(f(x,\alpha^{*,\varepsilon_{n}})+g(y,\beta^{*,\varepsilon_{n}})-\frac{1}{2}\|x\|^{2}-\frac{1}{2}\|y\|^{2}\Bigr)\biggr)\md \gamma.
	\end{equation*}
	Taking the limit \(\varepsilon_{n}\to 0\), by the dominated convergence theorem the above identity yields
	\begin{equation*}
		a_{i}= \lim_{n\to \infty} \int_{A_{i}(\hat{\alpha})\times \mathcal{Y}} \exp\biggl(\varepsilon_{n}^{-1}\Bigl(f(x,\hat{\alpha})+g(y,\hat{\beta})-\frac{1}{2}\|x\|^{2}-\frac{1}{2}\|y\|^{2}\Bigr)\biggr)\md \gamma.
	\end{equation*}
	For the  exponent above, we have for any \((x,y)\in A_{i}(\hat{\alpha})\times \mathcal{Y}\)
	\begin{align*}
		f(x,\hat{\alpha})+g(y,\hat{\beta})-\frac{1}{2}\|x\|^{2}-\frac{1}{2}\|y\|^{2} & = \langle x, x_{i} \rangle +\hat{\alpha}_{i} +\sup_{1\leq l\leq m} \{\langle y, y_{l} \rangle+\hat{\beta}_{l}\} -\frac{1}{2}\|x\|^{2}- \frac{1}{2}\|y\|^{2} \\
		                                                                             & < -\frac{1}{2}\|x-x_{i}\|^{2} -\sup_{1\leq l\leq m}\frac{1}{2}\|y-y_{l}\|^{2}<0.
	\end{align*}
	This leads to the contradiction that
	\begin{equation*}
		0<a_{i}= \lim_{n\to \infty}\int_{A_{i}(\hat{\alpha})\times \mathcal{Y}} \exp\biggl(\varepsilon_{n}^{-1}\Bigl(f(x,\hat{\alpha})+g(y,\hat{\beta})-\frac{1}{2}\|x\|^{2}-\frac{1}{2}\|y\|^{2}\Bigr)\biggr)\md \gamma=0.
	\end{equation*}
	Therefore, \eqref{eqn-lower} cannot be true, and we deduce \(\hat{\alpha}_{i}+\hat{\beta}_{j}= -\frac{1}{2}\|x_{i}\|^{2}-\frac{1}{2}\|y_{j}\|^{2}\), or equivalently, \((\hat{\alpha},\hat{\beta})\sim_{\oplus}(\alpha^{*,0},\beta^{*,0})\).

	\emph{Step 3.} We derive the limit of \(\pi^{*,\varepsilon}\).
	Following the same argument as Theorem \ref{thm-semi}, we have
	\begin{equation}
		\label{eqn-pi-eps}
		\md 	\pi^{*,\varepsilon}  =\exp\biggl(\varepsilon^{-1}\Bigl(f(x,\alpha^{*,\varepsilon})+g(y,\beta^{*,\varepsilon})-\frac{1}{2}\|x\|^{2}-\frac{1}{2}\|y\|^{2}\Bigr)\biggr) \md \gamma.
	\end{equation}
	Recall that we set \(\alpha^{*,0}_{i}= - \frac{1}{2}\|x_{i}\|^{2}\) and \(\beta^{*,0}_{j}= -\frac{1}{2}\|y_{j}\|^{2}\).
	We write  the expansions
	\begin{equation}
		\label{eqn-exp}
		\alpha^{*,\varepsilon}_{i}= \alpha^{*,0}_{i} - \frac{d   }{2} \varepsilon\ln \varepsilon+ \varepsilon p^{*,\varepsilon}_{i} \text{ and }\beta^{*,\varepsilon}_{j}= \beta^{*,0}_{j} - \frac{d }{2}  \varepsilon \ln \varepsilon+ \varepsilon q^{*,\varepsilon}_{i}
	\end{equation}
	and define
	\begin{align*}
		\pi^{*,\varepsilon}_{ij} & =\int_{A_{i}(\alpha^{*,\varepsilon})\times B_{j}(\beta^{*,\varepsilon})}\exp\biggl(\varepsilon^{-1}\Bigl(f(x,\alpha^{*,\varepsilon})+g(y,\beta^{*,\varepsilon})-\frac{1}{2}\|x\|^{2}-\frac{1}{2}\|y\|^{2}\Bigr)\biggr)\md \gamma \\
		                         & = \varepsilon^{-d}\int_{A_{i}(\alpha^{*,\varepsilon})\times B_{j}(\beta^{*,\varepsilon})}\exp\Bigl(-\frac{1}{2 \varepsilon} \|x-x_{i}\|^{2}- \frac{1}{2 \varepsilon}\|y-y_{j}\|^{2}\Bigr)\rho(x,y)\md x \md y                    \\
		                         & \quad \times \exp\Bigl(\frac{1}{\varepsilon}\Bigl(\alpha^{*,\varepsilon}+\beta^{*,\varepsilon}+\frac{1}{2 } \|x_{i}\|^{2}+ \frac{1}{2 }\|y_{j}\|^{2}\Bigr)+d \ln \varepsilon\Bigr)                                               \\
		                         & := \sigma_{ij}^{*,\varepsilon}\times \exp(p_{i}^{*,\varepsilon}+ q_{j}^{*,\varepsilon}).
	\end{align*}
	With the above notations, the first order optimality condition of  \eqref{eqn-dual-blowup} yields
	\begin{equation*}
		\left\{
		\begin{aligned}
			a_{i} & = \int_{A_{i}(\alpha^{*,\varepsilon})\times \mathcal{Y}} \exp\biggl(\varepsilon^{-1}\Bigl(f(x,\alpha^{*,\varepsilon})+g(y,\beta^{*,\varepsilon})-\frac{1}{2}\|x\|^{2}-\frac{1}{2}\|y\|^{2}\Bigr)\biggr)\md \gamma = \sum_{j=1}^{m}\sigma_{ij}^{*,\varepsilon}\exp(p_{i}^{*,\varepsilon}+ q_{j}^{*,\varepsilon}), \\
			b_{j} & = \int_{\mathcal{X}\times B_{j}(\beta^{*,\varepsilon})} \exp\biggl(\varepsilon^{-1}\Bigl(f(x,\alpha^{*,\varepsilon})+g(y,\beta^{*,\varepsilon})-\frac{1}{2}\|x\|^{2}-\frac{1}{2}\|y\|^{2}\Bigr)\biggr)\md \gamma=\sum_{j=1}^{m}\sigma_{ij}^{*,\varepsilon}\exp(p_{i}^{*,\varepsilon}+ q_{j}^{*,\varepsilon}).
		\end{aligned}
		\right.
	\end{equation*}
	In particular, by Theorem \ref{thm-se}, \((p^{*,\varepsilon}, q^{*,\varepsilon})\) is the unique solution to the Schrödinger equation \begin{equation*}
		\left\{
		\begin{aligned}
			a_{i} & =\sum_{j=1}^{m}\sigma_{ij}^{*,\varepsilon}\exp(p_{i}+q_{j}), \\
			b_{j} & =\sum_{i=1}^{n}\sigma_{ij}^{*,\varepsilon}\exp(p_{i}+q_{j}). \\
		\end{aligned}
		\right.
	\end{equation*}
	We notice that
	\begin{align*}
		\lim_{\varepsilon\to 0} \sigma_{ij}^{*,\varepsilon} & =\lim_{\varepsilon\to 0}\varepsilon^{-d}\int_{A_{i}(\alpha^{*,\varepsilon})\times B_{j}(\beta^{*,\varepsilon})}\exp\Bigl(-\frac{1}{2 \varepsilon} \|x-x_{i}\|^{2}- \frac{1}{2 \varepsilon}\|y-y_{j}\|^{2}\Bigr)\rho(x,y)\md x \md y \\
		                                                    & = \lim_{\varepsilon\to 0}\int_{U_{i}(\alpha^{*,\varepsilon})\times V_{j}(\beta^{*,\varepsilon})} \exp\Bigl(-\frac{1}{2}\|u\|^{2}-\frac{1}{2}\|v\|^{2}\Bigr)\rho(x_{0}+\sqrt{\varepsilon}u,y_{0}+\sqrt{\varepsilon}v) \md u\md v     \\
		                                                    & = \lim_{\varepsilon\to 0}\int_{\mathbb{R}^{d}\times \mathbb{R}^{d}} \exp\Bigl(-\frac{1}{2}\|u\|^{2}-\frac{1}{2}\|v\|^{2}\Bigr)\rho(x_{0}+\sqrt{\varepsilon}u,y_{0}+\sqrt{\varepsilon}v) \md u\md v                                  \\
		                                                    & = (2 \pi)^{d}\rho(x_{0},y_{0})>0,
	\end{align*}
	where we change the variable in the second line \(x-x_{0}=\sqrt{\varepsilon}u\), \(y-y_{0}=\sqrt{\varepsilon}v\) and set \(U_{i}(\alpha^{*},\varepsilon), V_{j}(\beta^{*,\varepsilon})\) be the corresponding domain.
	The third line follows from the fact that \(A_{i}(\alpha^{*,\varepsilon})\) converges to \(A_{i}(\alpha^{*,0})\) and \(B_{j}(\beta^{*,\varepsilon})\) converges to \(B_{j}(\beta^{*,0})\); the last line follows from Assumption~\ref{asmp-blowup}.
	By Remark \ref{rmk-stab}, we have the sequence of discrete measures \(\sum_{i=1}^{n}\sum_{j=1}^{m}\pi^{*,\varepsilon}_{ij}\delta_{(x_{i},y_{j})}\in \Pi(\mu,\nu)\) converges weakly to
	\begin{equation*}
		\pi^{*,0}=\argmin_{\pi\in \P(\mu,\nu)}H(\pi|\sigma),
	\end{equation*}
	which yields \(\lim_{\varepsilon\to 0}\pi^{*\varepsilon}_{ij}=\pi^{*,0}_{ij}\) for any \(1\leq i \leq n\) and \(1\leq j\leq m\).
	Therefore, for any continuous and bounded function \(\varphi\), plugging in to \eqref{eqn-pi-eps} we have
	\begin{align*}
		\lim_{\varepsilon\to 0}\int \varphi(x,y)\md \pi^{*,\varepsilon}(\md x ,\md y)=\lim_{\varepsilon\to 0} \sum_{i=1}^{n}\sum_{j=1}^{m} \pi^{*,\varepsilon}_{ij}\varphi(x_{i},y_{j})= \sum_{i=1}^{n}\sum_{j=1}^{m} \pi^{*,0}_{ij}\varphi(x_{i},y_{j}),
	\end{align*}
	which implies the weak convergence of \(\pi^{*,\varepsilon}\) to \(\pi^{*,0}\).

	\emph{Step 4.}
	We notice
	\begin{eqnarray}
		I^{\varepsilon}(\pi^{*,\varepsilon}) & =&U^{\varepsilon}(\alpha^{*,\varepsilon},\beta^{*,\varepsilon})                    \nonumber                                                                                 \\
		& =& \sum_{i=1}^{n}\varepsilon^{-1}a_{i}(\alpha_{i}^{*,\varepsilon}+\frac{1}{2}\|x_{i}\|^{2}) +\sum_{j=1}^{m}\varepsilon^{-1}b_{j}(\beta_{j}^{*,\varepsilon}+\frac{1}{2}\|y_{j}\|^{2}).\label{I_epsilon}
	\end{eqnarray}
	Plugging the expansions \eqref{eqn-exp} into the above identity yields
	\begin{align*}
		I^{\varepsilon}(\pi^{*,\varepsilon}) & = -d \ln \varepsilon  +\sum_{i=1}^{n}a_{i}p_{i}^{*,\varepsilon}+\sum_{j=1}^{m}b_{j}q_{j}^{*,\varepsilon} \\
		                                     & = -d \ln \varepsilon + \sum_{i=1}^{n}a_{i}p_{i}^{*,0}+\sum_{j=1}^{m}b_{j}q_{j}^{*,0} + o(1)              \\
		                                     & = -d \ln \varepsilon + H(\pi^{*,0}|\sigma) + o(1).
	\end{align*}
	The second last line follows from the convergence of \((p^{*,\varepsilon},q^{*,\varepsilon})\) to \((p^{*,0},q^{*,0})\) by Remark \ref{rmk-stab}.
\end{proof}

\section{Gradient ascent algorithm}
\label{sec-ga}
In this section, we investigate the gradient ascent method and establish a linear convergence rate.
Let \(\eta\) be a step size which will be specified later.
The gradient ascent algorithm is given iteratively by
\begin{equation}
	\label{eq:gradient-ascent}
	\left\{ \begin{aligned}
		\alpha^{t+1}_{i} & :=\alpha^{t}_{i}+\eta \partial_{\alpha^{t}_{i}}U(\alpha^{t},\beta^{t})= \alpha^{t}_{i} + \eta\biggl(a_{i}-\int_{A_{i}(\alpha^{t})\times \mathcal{Y}}\exp( f(x,\alpha^{t}) +g(y,\beta^{t}))  \tilde{\gamma}(\md x,\md y)\biggr), \\
		\beta_{j}^{t+1}  & :=\beta_{j}^{t}+\eta\partial_{\beta_{j}^{t}}U(\alpha^{t},\beta^{t})=\beta_{j}^{t}+\eta\biggl(b_{j}-\int_{\mathcal{X}\times B_{j}(\beta)^{t}} \exp(f(x,\alpha^{t}) +g(y,\beta^{t}))\tilde{\gamma}(\md x,\md y)\biggr),
	\end{aligned}\right.
\end{equation}
where \(\md \tilde{\gamma} = \exp(-\frac{1}{2}\|x\|^{2}-\frac{1}{2}\|y\|^{2})\md \gamma\).

\begin{asmp}
	\label{asmp-gamma}
	We assume that the measure \(\gamma\) admits a  continuous and bounded density \(\rho\). We define the weighted density \(\tilde{\rho}(x,y) := \exp\bigl(-\frac{1}{2}\|x\|^{2}-\frac{1}{2}\|y\|^{2}\bigr)\rho(x,y)\).
\end{asmp}

\begin{prop}
	\label{prop-gradient}
	Let Assumptions \ref{asmp-semi} and \ref{asmp-gamma} hold.
	For any initial point \((\alpha^{0},\beta^{0})\in \mathbb{R}^{n+m}\), there exists a constant \(\eta_{0}\) such that \(\lim_{t\to\infty}(\alpha^{t},\beta^{t})=(\alpha^{\infty},\beta^{\infty})\) exists for any step size \(0<\eta<\eta_{0}\).
	Moreover, we have \((\alpha^{\infty},\beta^{\infty})=(\alpha^{*},\beta^{*})\).
\end{prop}
\begin{proof}
	We claim that on the upper level set
	\[
		\{(\alpha,\beta):U(\alpha,\beta)\geq U(\alpha^{0},\beta^{0})\},
	\]
	\(\nabla U\) is \(L\)-Lipschitz for some constant \(L\) depending on the initial point \((\alpha^{0},\beta^{0})\).
	We note that \(A_{i}(\alpha)=\cap_{k\neq i}H_{ik}(\alpha_{i},\alpha_{k})\), where \(H_{ik}(\alpha_{i},\alpha_{k})=\{x:\langle x, x_{i}-x_{k} \rangle+(\alpha_{i}-\alpha_{k})\geq 0\}\) are half-spaces.
	Similarly, we write \(B_{j}(\beta)=\cap_{l\neq j} H_{jl}(\beta_{j},\beta_{l})\) where \(H_{jl}(\beta_{j},\beta_{l})=\{y:\langle y, y_{j}-y_{l} \rangle+(\beta_{j}-\beta_{l})\geq 0\}\).
	In particular, we deduce that for \(k\neq i\) and \(l\neq j\)
	\begin{equation}
		\label{eqn-derivative}
		\left\{
		\begin{aligned}
			\partial_{\alpha_{i}}^{2}U(\alpha,\beta)           & =-\int_{A_{i}(\alpha)\times \mathcal{Y}}\exp(f(x,\alpha) + g(y,\beta)) \tilde{\rho}(x,y)  \md x\md y -\sum_{k\neq i}\partial_{\alpha_{i}\alpha_{k}}^{2}U(\alpha,\beta), \\
			\partial_{\alpha_{i}\alpha_{k}}^{2}U(\alpha,\beta) & = \frac{1}{\|x_{i}-x_{k}\|}\int_{(A_{i}\cap A_{k})(\alpha)\times \mathcal{Y}}  \exp(f(x,\alpha) + g(y,\beta)) \tilde{\rho}(x,y) \md S_{ik}(x) \md y,                    \\
			\partial_{\beta_{j}}^{2}U(\alpha,\beta)            & =-\int_{\mathcal{X}\times B_{j}(\beta)}\exp(f(x,\alpha) + g(y,\beta)) \tilde{\rho}(x,y)\md x\md y - \sum_{l\neq j}\partial^{2}_{\beta_{j}\beta_{l}}U(\alpha,\beta),     \\
			\partial_{\beta_{j}\beta_{l}}^{2}U(\alpha,\beta)   & =\frac{1}{\|y_{j}-y_{l}\|}\int_{\mathcal{X}\times (B_{j}\cap B_{l})(\beta)} \exp(f(x,\alpha) + g(y,\beta)) \tilde{\rho}(x,y)\md x\md S_{jl}(y),                         \\
			\partial_{\alpha_{i}\beta_{j}}^{2}U(\alpha,\beta)  & = - \int_{A_{i}(\alpha)\times B_{j}(\beta)}\exp(f(x,\alpha) + g(y,\beta)) \tilde{\rho}(x,y)\md x\md y,
		\end{aligned}
		\right.
	\end{equation}
	where \(\md S_{ik}(x)\) and \(\md S_{jl}(y)\) denote the surface measures on the hyperplanes \(A_{i}\cap A_{l}\) and \(B_{j}\cap B_{l}\) respectively.
	By Assumption \ref{asmp-gamma}, \(\rho\) is bounded and hence on the upper level set all second derivatives are bounded.
	Therefore, \(\nabla U\) is \(L\)-Lipschitz.

	Now, we take \(\eta_{0}= 2 L^{-1} \).
	Without loss of generality we assume \(\nabla U(\alpha^{0},\beta^{0})\neq 0\).
	We write \((\alpha^{\lambda},\beta^{\lambda})=(\alpha^{0},\beta^{0})+\lambda \eta \nabla U(\alpha^{0},\beta^{0})\) to simplify the notation.
	We claim that \(U(\alpha^{1},\beta^{1})> U(\alpha^{0},\beta^{0})\).
	Otherwise, there exists \(\lambda\in(0,1)\) such that \(U(\alpha^{\lambda},\beta^{\lambda})= U(\alpha^{0},\beta^{0})\).
	Recall that \(\nabla U\) is \(L\)-Lipschitz on the line segment \([(\alpha^{0},\beta^{0}),(\alpha^{\lambda},\beta^{\lambda}) ]\).
	Therefore, we reach a contradiction from the estimate
	\begin{align*}
		0=U(\alpha^{\lambda},\beta^{\lambda})-U(\alpha^{0},\beta^{0})&\geq  \lambda \eta\|\nabla U(\alpha^{0},\beta^{0})\|^{2} -\frac{L}{2}\lambda^{2}\eta^{2}\|\nabla U(\alpha^{0},\beta^{0})\|^{2}>0.	\end{align*}
	By induction, we show that  \(\{U(\alpha^{t},\beta^{t})\}_{t\geq 0}\)  is an increasing sequence.
	Moreover, the limit of any converging subsequence of  \(\{(\alpha^{t},\beta^{t})\}_{t\geq 0}\), \((\alpha^{\infty},\beta^{\infty})\), satisfies \(\nabla U(\alpha^{\infty},\beta^{\infty})=0\).
	Hence, from the uniqueness of the dual optimizer in Theorem \ref{thm-semi}, we deduce 	\((\alpha^{\infty},\beta^{\infty})\sim_{\oplus} (\alpha^{*},\beta^{*})\) and complete the proof.
\end{proof}

In the following proposition, we establish the linear convergence of the gradient ascent algorithm.
\begin{prop}
	\label{prop-ga}
	Let Assumptions \ref{asmp-semi} and
	\ref{asmp-gamma} hold.
	For any initial point \((\alpha^{0},\beta^{0})\in \mathbb{R}^{n+m}\), there exists a constant \(\eta_{0}\)  such that for any step size \(0<\eta<\eta_{0}\) we have the linear convergence 	\begin{equation*}
		\|(\alpha^{t},\beta^{t})-(\alpha^{*},\beta^{*})\|_{l^{2}_{\oplus}}=O(\theta^{t}) \text{ as } t\to\infty
	\end{equation*}
	for some \(\theta\in(0,1)\).
\end{prop}
\begin{proof}
	\emph{Step 1.}
	We write \((u^{t},v^{t})=(\alpha^{t},\beta^{t})-(\alpha^{*},\beta^{*})\) and \((\alpha^{t,\lambda},\beta^{t,\lambda})=(\alpha^{*},\beta^{*})+\lambda (u^{t},v^{t})\).
	By the fundamental theorem of calculus, we note that \((u^{t+1},v^{t+1})=\Phi_{t}(u^{t},v^{t})\) where \(\Phi_{t}\) is a linear map given by
	\begin{equation*}
		\Phi_{t}(u,v):=\left( \begin{aligned}
				u_{i} +\eta \sum_{k=1}^{n} \biggl(\int_{0}^{1} \partial^{2}_{\alpha_{i}\alpha_{k} }U(\alpha^{t,\lambda},\beta^{t,\lambda})\md \lambda\biggr)  u_{k} +\eta \sum_{l=1}^{m} \biggl(\int_{0}^{1} \partial^{2}_{\alpha_{i}\beta_{l} }U(\alpha^{t,\lambda},\beta^{t,\lambda})\md \lambda\biggr)  v_{l} \\
				v_{j} +\eta \sum_{k=1}^{m} \biggl(\int_{0}^{1} \partial^{2}_{\beta_{j}\alpha_{k} }U(\alpha^{t,\lambda},\beta^{t,\lambda})\md \lambda\biggr)  u_{k} +\eta \sum_{l=1}^{m} \biggl(\int_{0}^{1} \partial^{2}_{\beta_{j}\beta_{l} }U(\alpha^{t,\lambda},\beta^{t,\lambda})\md \lambda\biggr)  v_{l}   \\
			\end{aligned}\right).
	\end{equation*}
	Similarly, we define
	\begin{equation}
		\label{eqn-phi}
		\Phi_{*}(u,v):=\left( \begin{aligned}
				u_{i} +\eta \sum_{k=1}^{n}  \partial^{2}_{\alpha_{i}\alpha_{k} }U(\alpha^{*},\beta^{*}) u_{k} +\eta \sum_{l=1}^{m} \partial^{2}_{\alpha_{i}\beta_{l} }U(\alpha^{*},\beta^{*}) v_{l} \\
				v_{j} +\eta \sum_{k=1}^{m}  \partial^{2}_{\beta_{j}\alpha_{k} }U(\alpha^{*},\beta^{*}) u_{k} +\eta \sum_{l=1}^{m}  \partial^{2}_{\beta_{j}\beta_{l} }U(\alpha^{*},\beta^{*})  v_{l}
			\end{aligned}\right).
	\end{equation}
	By Proposition \ref{prop-gradient}, for sufficiently small step size \(\varepsilon\), we have \(\lim_{t\to\infty}\|(\alpha^{t},\beta^{t})-(\alpha^{*},\beta^{*})\|_{l^{2}_{\oplus}}=0\).
	Therefore, from \eqref{eqn-derivative} the Hessian \(\nabla^{2}U(\alpha^{t,\lambda},\beta^{t,\lambda})\) converges to \(\nabla^{2} U(\alpha^{*},\beta^{*})\) which further implies
	\begin{equation}
		\label{eqn-op}
		\lim_{t\to\infty}\sup_{\|(u,v)\|_{l^{2}_{\oplus}}\leq 1}\|\Phi_{t}(u,v)-\Phi_{*}(u,v)\|_{l^{2}_{\oplus}}=0.
	\end{equation}

	\emph{Step 2.} We claim that there exists \(\delta>0\) and \(M>0\) such that
	\begin{equation*}
		\|\nabla^{2}U(\alpha^{*},\beta^{*})(u,v)^{\intercal}\|_{l^{2}_{\oplus}}\leq M \|(u,v)\|_{l^{2}_{\oplus}}, \quad
		(u,v)\nabla^{2}U(\alpha^{*},\beta^{*})(u,v)^{\intercal}\leq -\delta \|(u,v)\|^{2}_{l^{2}_{\oplus}}.
	\end{equation*}
	The first estimate follows directly from the fact that the second derivatives of \(U\) at \((\alpha^{*},\beta^{*})\) are bounded, and hence \(\nabla^{2}U(\alpha^{*},\beta^{*})\)  is continuous under \(\|\cdot\|_{l^{2}_{\oplus}}\).
	For the second estimate, we compute from \eqref{eqn-derivative} that
	\begin{align*}
		(u,v)\nabla^{2}U(\alpha^{*},\beta^{*})(u,v)^{\intercal} & =\sum_{i=1}^{n}\sum_{j=1}^{m}\partial^{2}_{\alpha_{i}\beta_{j}}U(\alpha^{*},\beta^{*}) (u_{i}+v_{j})^{2}                   \\
		                                                        & - 2\sum_{1\leq i< k\leq n} \partial^{2}_{\alpha_{i}\alpha_{k}}U(\alpha^{*},\beta^{*}) (u_{i}-u_{k})^{2}- 2\sum_{1\leq j< l
		\leq m} \partial^{2}_{\beta_{j}\beta_{l}}U(\alpha^{*},\beta^{*}) (v_{j}-v_{l})^{2}                                                                                                   \\
		                                                        & \leq \sum_{i=1}^{n}\sum_{j=1}^{m}\partial^{2}_{\alpha_{i}\beta_{j}}U(\alpha^{*},\beta^{*}) (u_{i}+v_{j})^{2}.
	\end{align*}
	From the first order condition \eqref{eqn-se}, we notice that \(\gamma(A_{i}(\alpha^{*}))>0\) and \(\gamma(B_{j}(\beta^{*}))>0\) for any \(1\leq i,j\leq n\).
	Therefore,  we have \(\partial^{2}_{\alpha_{i},\beta_{j}}U(\alpha^{*},\beta^{*})\leq -\delta\) for sufficiently small \(\delta\).
	Plugging into the above estimate, we derive the upper bound estimate
	\begin{equation*}
		(u,v)\nabla^{2}U(\alpha^{*},\beta^{*})(u,v)^{\intercal} \leq -\delta\|(u,v)\|^{2}_{l^{2}_{\oplus}}.
	\end{equation*}

	\emph{Step 3.} 	Let \(\eta_{0}\) be the constant in Proposition \ref{prop-gradient}. We take \(\eta < \min\{\frac{2 \delta}{M^{2}},\eta_{0}\}\) and \(\theta_{0} := 1- \eta M^{2}(\frac{2 \delta}{M^{2}}-\eta)\in (0,1)\).
	We notice that
	\begin{equation*}
		\Phi_{*}(u,v)=(\Id +\eta\nabla^{2}U(\alpha^{*},\beta^{*}))(u,v)^{\intercal},
	\end{equation*}
	which yields that
	\begin{align*}
		\|\Phi_{*}(u,v)\|_{l^{2}_{\oplus}}^{2} & = \|(u,v)\|_{l^{2}_{\oplus}}^{2} -2 \eta(u,v)\nabla^{2}U(\alpha^{*},\beta^{*})(u,v)^{\intercal} + \eta^{2} \|\nabla^{2}U(\alpha^{*},\beta^{*})(u,v)^{\intercal}\|_{l^{2}_{\oplus}}^{2} \\
		                                       & \leq (1- 2 \eta \delta + \eta^{2} M^{2}) \|(u,v)\|_{l^{2}_{\oplus}}^{2}=\theta_{0} \|(u,v)\|_{l^{2}_{\oplus}}^{2}.
	\end{align*}
	In particular, \(\Phi_{*}\) is a contraction under \(\|\cdot\|_{l^{2}_{\oplus}}\).
	Since \(\Phi_{t}\) converges to \(\Phi_{*}\) in the operator norm by \eqref{eqn-op}, we have \(\Phi_{t}\) is a contraction for sufficiently large \(t\).
	Therefore, we obtain the linear convergence of  the gradient ascent algorithm, i.e.,
	\begin{equation*}
		\|(\alpha^{t},\beta^{t})-(\alpha^{*},\beta^{*})\|_{l^{2}_{\oplus}}=O(\theta^{t}) \text{ as } t\to\infty,
	\end{equation*}
	for some \(\theta<\theta_{0}\).
\end{proof}

\section{Sinkhorn-type algorithm}
\label{sec-sinkhorn}
In this section, we propose a Sinkhorn-type algorithm to fast compute the solution of the dual problem.
We recall from the proof Theorem \ref{thm-semi}, a sufficient and necessary condition for the dual optimizer is the first order condition \eqref{eqn-se}:
\begin{equation}
	\left\{
	\begin{aligned}
		\partial_{\alpha_{i}}U(\alpha^{*},\beta^{*})=a_{i}-\int_{A_{i}(\alpha^{*})\times \mathcal{Y}}\exp\Bigl(f(x,\alpha^{*})+g(y,\beta^{*})-\frac{1}{2}\|x\|^{2}-\frac{1}{2}\|y\|^{2}\Bigr)\md \gamma & =0, \\
		\partial_{\beta_{j}}U(\alpha^{*},\beta^{*})=b_{j}-\int_{\mathcal{X}\times B_{j}(\beta^{*})}\exp\Bigl(f(x,\alpha^{*})+g(y,\beta^{*})-\frac{1}{2}\|x\|^{2}-\frac{1}{2}\|y\|^{2}\Bigr)\md \gamma   & =0.
	\end{aligned}
	\right.
\end{equation}
In contrast to the classical Sinkhorn algorithm, the above equations do not admit an explicit closed-form iteration and involve the computation of the Laguerre cells \(A_{i}(\alpha)\) and \(B_{j}(\beta)\).
To address this issue, we introduce an approximated dual objective
\begin{equation*}
	\label{eqn-reg}
	U_{\lambda}(\alpha,\beta):=\sum_{i=1}^{n}a_{i}(\alpha_{i}+\frac{1}{2}\|x_{i}\|^{2})+\sum_{j=1}^{m}b_{j}(\beta_{j}+\frac{1}{2}\|y_{j}\|^{2})-\int\exp(\phi_{\lambda}(x,\alpha)+\psi_{\lambda}(y,\beta))\md \tilde{\gamma}+1,
\end{equation*}
where \(\phi_{\lambda}(x,\alpha)=\bigl(\sum_{i=1}^{n}\exp(\lambda \langle x, x_{i} \rangle+\lambda \alpha_{i})\bigr)^{1/\lambda}\) and \(\psi_{\lambda}(y,\beta)=\bigl(\sum_{j=1}^{m}\exp(\lambda \langle y, y_{j} \rangle+\lambda \beta_{j})\bigr)^{1/\lambda}\).
It is clear that \(U_{\lambda}\) converges to \(U\) as \(\lambda\) goes to infinity.
Its associated first order conditions read as
\begin{equation}
	\label{eqn-fo}
	\left\{	\begin{aligned}
		a_{i} & = \int \phi_{\lambda}(x,\alpha)^{1-\lambda} \psi_{\lambda}(y,\beta)  \exp(\lambda \langle x, x_{i} \rangle +  \lambda \alpha_{i}) \tilde{\gamma}(\md x,\md y), \\
		b_{j} & = \int \phi_{\lambda}(x,\alpha) \psi_{\lambda}(y,\beta)^{1-\lambda}  \exp(\lambda \langle y, y_{j} \rangle +  \lambda \beta_{j}) \tilde{\gamma}(\md x,\md y).
	\end{aligned}
	\right.
\end{equation}
We define  map \((\alpha,\beta)\mapsto \Psi(\alpha,\beta)=(\alpha',\beta')\) as
\begin{equation*}
	\left\{	\begin{aligned}
		\alpha_{i}' & :=\frac{1}{\lambda}\log(a_{i})-\frac{1}{\lambda}\log\Biggl(\int \phi_{\lambda}(x,\alpha)^{1-\lambda} \psi_{\lambda}(y,\beta)  \exp(\lambda \langle x, x_{i} \rangle) \tilde{\gamma}(\md x,\md y)\Biggr), \\
		\beta_{j}'  & :=\frac{1}{\lambda}\log(b_{j})-\frac{1}{\lambda}\log\Biggl(\int \phi_{\lambda}(x,\alpha) \psi_{\lambda}(y,\beta)^{1-\lambda}  \exp(\lambda \langle y, y_{j} \rangle) \tilde{\gamma}(\md x,\md y)\Biggr).
	\end{aligned}
	\right.
\end{equation*}
In the next proposition we show that \(\Psi\) has a fixed point.
Recall we write the maximum and minimum of a vector \(v\) as \(\overline{v}=\max_{i}v_{i}\), \(\underline{v}=\min_{i}v_{i}\)
\begin{prop}
	\label{prop-sinkhorn}
	Let Assumption \ref{asmp-semi} hold and \(\lambda>2\).
	There exists \(M>0\) such that  the restriction of \(\Psi\) on \(B_{R}:=\{(\alpha,\beta):-R\leq \underline{\alpha}+\underline{\beta}\leq  \overline{\alpha}+\overline{\beta}\leq R\}\) is a contraction under \(\|\cdot\|_{l^{\infty}_{\oplus}}\) for any \(R\geq M\).
	In particular, there exists \(\theta\in(0,1)\) and a fixed point \((\alpha^{\lambda},\beta^{\lambda})\) of \(\Psi\) such that,  for any \((\alpha,\beta)\in \mathbb{R}^{n+m}\),
	\begin{equation*}
		\|\Psi^{(t)}(\alpha,\beta)-(\alpha^{\lambda},\beta^{\lambda})\|_{l^{\infty}_{\oplus}}= O(\theta^{t}) \text{ as }t\to\infty,
	\end{equation*}
	where \(\Psi^{(t)}\) is the \(t\)-th iteration of \(\Psi\).
\end{prop}
\begin{proof}
	\emph{Step 1.} We first show that \(\Psi|_{B_{R}}\) maps into \(B_{R}\) for any \(R\geq M\), where we take
	\begin{align*}
		M & =  |\log(\underline{a})| + |\log(\underline{b})|                                                                                                                                                                    \\
		  & \quad + \biggl|\log\biggl(\int \phi_{\lambda}(x,\mathbf{0})^{1-\lambda}\exp( \lambda\inf_{1\leq i \leq n}\langle x, x_{i} \rangle+\inf_{1\leq j \leq m}\langle y, y_{j} \rangle)  \md \tilde{\gamma}\biggr)\biggr|  \\
		  & \quad + \biggl|\log\biggl(\int \psi_{\lambda}(y,\mathbf{0})^{1-\lambda}\exp( \inf_{1\leq i \leq n}\langle x, x_{i} \rangle + \lambda\inf_{1\leq j \leq m}\langle y, y_{j} )\rangle \md \tilde{\gamma}\biggr)\biggr| \\
		  & \quad + \biggl|\log \biggl(\int \psi_{\lambda}(y,\mathbf{0})\exp((1-\lambda)\inf_{1\leq i \leq n}\langle x, x_{i} \rangle+\lambda \sup_{1\leq j \leq m}\langle y, y_{j} \rangle)\md \tilde{\gamma}\biggr)\biggr|    \\
		  & \quad + \biggl|\log \biggl(\int \phi_{\lambda}(x,\mathbf{0})\exp(\lambda\inf_{1\leq i \leq n}\langle x, x_{i} \rangle+(1-\lambda )\sup_{1\leq j \leq m}\langle y, y_{j} \rangle)\md \tilde{\gamma}\biggr)\biggr|.
	\end{align*}
	We notice that
	\begin{equation*}
		\phi_{\lambda}(x,\alpha)\leq \phi_{\lambda}(x,\overline{\alpha}\mathbf{1}) \text{ and } \psi_{\lambda}(y,\beta)\geq \exp(\inf_{1\leq j \leq m}\langle y, y_{j} \rangle+ \overline{\beta}).
	\end{equation*}
	This yields
	\begin{align*}
		\overline{\alpha}' & \leq \frac{1}{\lambda}\log{\overline{a}} -\frac{1}{\lambda} \log\biggl(\int \phi_{\lambda}(x,\overline{\alpha}\mathbf{1})^{1-\lambda}\exp(\overline{\beta}+\inf_{1\leq j \leq m}\langle y, y_{j} \rangle + \lambda\inf_{1\leq i \leq n}\langle x, x_{i} \rangle) \md \tilde{\gamma}\biggr)                   \\
		                   & \leq -\frac{1}{\lambda} \log\biggl(\int \phi_{\lambda}(x,\mathbf{0})^{1-\lambda}\exp(\inf_{1\leq j \leq m}\langle y, y_{j} \rangle + \lambda\inf_{1\leq i \leq n}\langle x, x_{i} \rangle) \md \tilde{\gamma}\biggr) + \Bigl(1-\frac{1}{\lambda}\Bigr) \overline{\alpha} - \frac{1}{\lambda}\overline{\beta} \\
		                   & \leq \frac{1}{\lambda}M + \Bigl(1-\frac{1}{\lambda}\Bigr) \overline{\alpha}- \frac{1}{\lambda}\overline{\beta}.
	\end{align*}
	Similarly, we have \(\overline{\beta}'\leq \frac{1}{\lambda}M -\frac{1}{\lambda}\overline{\alpha}+ (1-\frac{1}{\lambda})\overline{\beta}\).
	Summing up above estimates we derive
	\begin{equation*}
		\overline{\alpha}'+\overline{\beta}'\leq \frac{2}{\lambda}M + \Bigl(1-\frac{2}{\lambda}\Bigr)(\overline{\alpha}+\overline{\beta})\leq R.
	\end{equation*}
	On the other hand, we notice that
	\begin{equation*}
		\phi_{\lambda}(x,\alpha)\geq \exp(\inf_{1\leq i \leq n}\langle x, x_{i} \rangle+\overline{\alpha}) \text{ and } \psi_{\lambda}(y,\beta)\leq \psi_{\lambda}(y,\overline{\beta}\mathbf{1}),
	\end{equation*}
	which implies
	\begin{align*}
		\underline{\alpha}' & \geq \frac{1}{\lambda}\log(\underline{a})- \frac{1}{\lambda}\log\biggl(\int \exp((1-\lambda)\inf_{1\leq i \leq n}\langle x, x_{i} \rangle+(1-\lambda)\overline{\alpha})\psi_{\lambda}(y,\overline{\beta}\mathbf{1})\exp(\lambda\sup_{1\leq j \leq m}\langle y, y_{j} \rangle)\md \tilde{\gamma}\biggr)                                    \\
		                    & \geq \frac{1}{\lambda}\log(\underline{a})- \frac{1}{\lambda}\log\biggl(\int \psi_{\lambda}(y,\mathbf{0})\exp((1-\lambda)\inf_{1\leq i \leq n}\langle x, x_{i} \rangle+\lambda\sup_{1\leq j \leq m}\langle y, y_{j} \rangle)\md \tilde{\gamma}\biggr) + \Bigl(1-\frac{1}{\lambda}\Bigr)\overline{\alpha}-\frac{1}{\lambda}\overline{\beta} \\
		                    & \geq -\frac{1}{\lambda}M + \Bigl(1-\frac{1}{\lambda}\Bigr)\overline{\alpha}-\frac{1}{\lambda}\overline{\beta}.
	\end{align*}
	Similarly, we have \(\underline{\beta}'\geq -\frac{1}{\lambda}M -\frac{1}{\lambda}\overline{\alpha}+(1-\frac{1}{\lambda})\overline{\beta}\).
	Therefore, we show that
	\[
		\underline{\alpha}'+\underline{\beta}'\geq -\frac{2}{\lambda}M+\Bigl(1-\frac{2}{\lambda}\Bigr)(\underline{\alpha}+\underline{\beta})\geq -R.
	\]

	\emph{Step 2.} Let \((\alpha,\beta), (\tilde{\alpha},\tilde{\beta})\in B_{R}\).
	We write \((u,v)= (\tilde{\alpha},\tilde{\beta})-(\alpha,\beta)\) and \((u',v')=\Psi((\tilde{\alpha},\tilde{\beta}))-\Psi((\alpha,\beta))\).
	We claim that there exists \(\eta>0\) such that
	\begin{equation*}
		\begin{pmatrix}
			u' \\
			v'
		\end{pmatrix}
		= \begin{pmatrix}
			P & Q \\
			R & S
		\end{pmatrix}
		\begin{pmatrix}
			u \\
			v,
		\end{pmatrix}
	\end{equation*}
	where \(P,\,Q,\,R,\,S\) are row-stochastic matrices with entries no less than \(\frac{1}{\eta}\).
	Equivalently, it means that
	\begin{equation}
		\label{eqn-uv}
		u'_{k}=\Bigl(1-\frac{1}{\lambda}\Bigr)\sum_{i=1}^{n}p_{ki}u_{i} -\frac{1}{\lambda}\sum_{j=1}^{m}q_{kj}v_{j} \text{ and } v'_{l}=-\frac{1}{\lambda}\sum_{i=1}^{n}r_{li}u_{i}+\Bigl(1-\frac{1}{\lambda}\Bigr)\sum_{j=1}^{m}s_{lj}v_{j},
	\end{equation}
	where \(\sum_{i}^{n}p_{ki}=\sum_{j=1}^{m}q_{kj}=\sum_{i=1}^{n}r_{li}=\sum_{j=1}^{m}s_{lj}=1\) and \(\min_{i,j,k,l}\{p_{ki},q_{kj},r_{li},s_{lj}\}\geq \frac{1}{\eta}\).
	We define
	\begin{equation*}
		p_{ki}(\alpha,\beta)=\frac{\int \phi_{\lambda}(x,\alpha)^{1-2\lambda}\psi_{\lambda}(y,\beta)\exp(\langle x, x_{k} +\lambda x_{i}\rangle + \lambda \alpha_{i})\md \tilde{\gamma}}{\int \phi_{\lambda}(x,\alpha)^{1-\lambda}\psi_{\lambda}(y,\beta)\exp(\langle x, x_{k} \rangle)\md \tilde{\gamma}}>0
	\end{equation*}
	and
	\begin{equation*}
		q_{kj}(\alpha,\beta)= \frac{\int \phi_{\lambda}(x,\alpha)^{1-\lambda}\psi_{\lambda}(y,\beta)^{1-\lambda}\exp(\langle x, x_{k}\rangle+ \lambda \langle y, y_{j} \rangle + \lambda \beta_{j})\md \tilde{\gamma}}{\int \phi_{\lambda}(x,\alpha)^{1-\lambda}\psi_{\lambda}(y,\beta)\exp(\langle x, x_{k} \rangle)\md \tilde{\gamma}}>0.
	\end{equation*}
	By the fundamental theorem of calculus, we have
	\begin{align*}
		u_{k}' & = \Psi_{k}(\tilde{\alpha},\tilde{\beta})- \Psi_{k}(\alpha, \beta)                                                                                                                                                                                    \\
		       & = \int_{0}^{1} \sum_{i=1}^{n} \partial_{\alpha_{i}}\Psi_{k}(\alpha+\varepsilon u,\beta+\varepsilon v) u_{i} +\sum_{j=1}^{m} \partial_{\beta_{j}}\Psi_{k}(\alpha+\varepsilon u,\beta+\varepsilon v) v_{j} \md \varepsilon                             \\
		       & = \Bigl(1-\frac{1}{\lambda}\Bigr)\sum_{i=1}^{n}\int_{0}^{1}p_{ki}(\alpha+\varepsilon u,\beta+\varepsilon v)\md \varepsilon u_{i}- \frac{1}{\lambda} \sum_{j=1}^{m}\int_{0}^{1}q_{kj}(\alpha+\varepsilon u,\beta+\varepsilon v)\md \varepsilon v_{j}.
	\end{align*}
	Hence, we take \(p_{ki}=\int_{0}^{1}p_{ki}(\alpha+\varepsilon u,\beta+\varepsilon v)\md \varepsilon \) and \(q_{kj}=\int_{0}^{1}q_{kj}(\alpha+\varepsilon u,\beta+\varepsilon v)\md \varepsilon \).
	It is direct to verify \(\sum_{i=1}^{n}p_{ki}=\sum_{j=1}^{m}q_{kj}=1\).
	As \(p_{ki}(\cdot)\) and \(q_{kj}(\cdot)\) are bounded away from zero  on \(B_{R}\), we verify \(p_{ki}>\frac{1}{\eta}\) and \(q_{kj}>\frac{1}{\eta}\) for sufficiently large \(\eta\).
	Hence, we show the claim \eqref{eqn-uv} holds for \(u'\), and it also holds for \(v'\) following the same arguments.

	\emph{Step 3.} We show that \(\Psi|_{B_{R}}\) is a contraction under \(\|\cdot\|_{l^{\infty}_{\oplus}}\).
	By \eqref{eqn-uv}, we can bound \(u'\) and \(v'\) by
	\begin{equation*}
		\left\{
		\begin{aligned}
			\overline{u}'  & \leq \Bigl(1-\frac{1}{\lambda}\Bigr)\Bigl(\frac{1}{\eta}\underline{u}+ \Bigl(1-\frac{1}{\eta}\Bigr)\overline{u}\Bigr) -\frac{1}{\lambda}\Bigl(\frac{1}{\eta} \overline{v}+\Bigl(1-\frac{1}{\eta}\Bigr)\underline{v}\Bigr),  \\
			\underline{u}' & \geq \Bigl(1-\frac{1}{\lambda}\Bigr)\Bigl(\frac{1}{\eta} \overline{u}+ \Bigl(1-\frac{1}{\eta}\Bigr)\underline{u}\Bigr) -\frac{1}{\lambda}\Bigl(\frac{1}{\eta} \underline{v}+\Bigl(1-\frac{1}{\eta}\Bigr)\overline{v}\Bigr), \\
			\overline{v}'  & \leq -\frac{1}{\lambda}\Bigl(\frac{1}{\eta}\overline{u}+ \Bigl(1-\frac{1}{\eta}\Bigr)\underline{u}\Bigr) +\Bigl(1-\frac{1}{\lambda}\Bigr)\Bigl(\frac{1}{\eta} \underline{v}+\Bigl(1-\frac{1}{\eta}\Bigr)\overline{v}\Bigr), \\
			\underline{v}' & \geq -\frac{1}{\lambda}\Bigl(\frac{1}{\eta} \underline{u}+ \Bigl(1-\frac{1}{\eta}\Bigr)\overline{u}\Bigr) +\Bigl(1-\frac{1}{\lambda}\Bigr)\Bigl(\frac{1}{\eta}\overline{v}+\Bigl(1-\frac{1}{\eta}\Bigr)\underline{v}\Bigr).
		\end{aligned}
		\right.
	\end{equation*}
	Together with \(\|(u,v)\|_{l^{\infty}_{\oplus}}=\max\{\overline{u}+\overline{v},-\underline{u}-\underline{v}\}\), we obtain
	\begin{equation*}
		\left\{
		\begin{aligned}
			\overline{u}'+\overline{v}'   & \leq \Bigl(1-\frac{1}{\lambda}-\frac{1}{\eta}\Bigr)(\overline{u}+\overline{v})+\Bigl(\frac{1}{\eta}-\frac{1}{\lambda}\Bigr)(\underline{u}+\underline{v})\leq \Bigl(1-\frac{2}{\max\{\lambda,\eta\}}\Bigr)\|(u,v)\|_{l^{\infty}_{\oplus}}   \\
			\underline{u}'+\underline{v}' & \geq \Bigl(1-\frac{1}{\lambda}-\frac{1}{\eta}\Bigr)(\underline{u}+\underline{v})+\Bigl(\frac{1}{\eta}-\frac{1}{\lambda}\Bigr)(\overline{u}+\overline{v})\geq -\Bigl(1-\frac{2}{\max\{\lambda,\eta\}}\Bigr)\|(u,v)\|_{l^{\infty}_{\oplus}}.
		\end{aligned}
		\right.
	\end{equation*}
	Therefore, we deduce
	\begin{align*}
		\|\Psi(\tilde{\alpha},\tilde{\beta})-\Psi(\alpha,\beta)\|_{l^{\infty}_{\oplus}}  =\|(u',v')\|_{l^{\infty}_{\oplus}} & =\max\{\overline{u}'+\overline{v}',-\underline{u}'-\underline{v}'\}                                                    \\
		                                                                                                                    & \leq \Bigl(1-\frac{2}{\max\{\lambda,\eta\}}\Bigr)\|(u,v)\|_{l^{\infty}_{\oplus}}                                       \\
		                                                                                                                    & = \Bigl(1-\frac{2}{\max\{\lambda,\eta\}}\Bigr)\|(\tilde{\alpha},\tilde{\beta})-(\alpha,\beta)\|_{l^{\infty}_{\oplus}}.
	\end{align*}

	\emph{Step 4.} For any \((\alpha,\beta)\), there exists \(R\geq M\) such that \((\alpha,\beta)\in B_{R}\).
	As \(\Psi|_{B_{R}}\) is a contraction, we have \(\Psi^{(t)}(\alpha,\beta)\) converges to a limit \((\alpha^{\lambda},\beta^{\lambda})\) with a linear rate under \(\|\cdot\|_{l^{\infty}_{\oplus}}\).
	We note that the convergence rate depends on \(\eta\) and hence depends on the choice of the initial point \((\alpha,\beta)\).
	However, for any \((\alpha,\beta)\) after sufficiently many steps of iterations, \(\Psi^{(t)}(\alpha,\beta)\) is close to the \((\alpha^{\lambda},\beta^{\lambda})\), so the corresponding  \(\eta\) will be close to the one for \((\alpha^{\lambda},\beta^{\lambda})\).
	Therefore, there exists a universal asymptotic liner convergence rate \(\theta\) for any initial point \((\alpha,\beta)\).
	We further show that \((\alpha^{\lambda},\beta^{\lambda})\) is actually a fixed point of \(\Psi\).
	Since the limit \((\alpha^{\lambda},\beta^{\lambda})\) is a fixed point of the induced map \(\Psi_{\oplus}:\mathbb{R}^{n+m}_{\oplus}\to \mathbb{R}^{n+m}_{\oplus}\),  there exists \(r\in \mathbb{R}\) such that
	\begin{equation*}
		\Psi(\alpha^{\lambda},\beta^{\lambda})=(\alpha^{\lambda},\beta^{\lambda})+r(\mathbf{1},-\mathbf{1}).
	\end{equation*}
	This implies
	\begin{equation*}
		\left\{
		\begin{aligned}
			1=\sum_{i=1}^{n}a_{i}=\exp(r)\int \phi_{\lambda}(x,\alpha^{\lambda})\psi_{\lambda}(y,\beta^{\lambda}) \md \tilde{\gamma}, \\
			1=\sum_{j=1}^{m}b_{i}=\exp(-r)\int \phi_{\lambda}(x,\alpha^{\lambda})\psi_{\lambda}(y,\beta^{\lambda}) \md \tilde{\gamma}.
		\end{aligned}
		\right.
	\end{equation*}
	Therefore, we obtain \(r=0\) and \((\alpha^{\lambda},\beta^{\lambda})\) is a fixed point of \(\Psi\).
\end{proof}

\begin{prop}
	Let Assumption \ref{asmp-semi} hold.
	The regularized problem \eqref{eqn-reg} has a unique optimizer \((\alpha^{\lambda},\beta^{\lambda})\) in \(\mathbb{R}^{n+m}_{\oplus}\).
	Moreover, we have \(\lim_{\lambda\to \infty}\|(\alpha^{\lambda},\beta^{\lambda})- (\alpha^{*},\beta^{*})\|_{l^{\infty}_{\oplus}}=0\).
\end{prop}
\begin{proof}
	By Proposition \ref{prop-sinkhorn}, for any \(R\geq M\), \(\Psi|_{B_{R}}\) is a contraction in \(\|\cdot\|_{l^{\infty}_{\oplus}}\), which implies the uniqueness of the optimizer.
	Since \(\phi_{\lambda}(x,\alpha) \searrow f(x,\alpha)\) and \(\psi_{\lambda}(y,\beta) \searrow g(y,\beta)\) as \(\lambda\) goes to infinity,
	we have
	\(\lim_{\lambda\to\infty }U_{\lambda}(\alpha^{*},\beta^{*})=U(\alpha^{*},\beta^{*})\) from the monotone convergence theorem.
	Together with \(U(\alpha^{\lambda},\beta^{\lambda})\geq U_{\lambda}(\alpha^{*},\beta^{*})\), we deduce that \(\{(\alpha^{\lambda},\beta^{\lambda})\}_{\lambda\geq 2}\) is contained in an upper level set of \(U\), and hence it forms a precompact set in \(\mathbb{R}^{n+m}_{\oplus}\) following the same argument in \emph{Step 1} of the proof of Theorem~\ref{thm-semi}.
	Recall \((\alpha^{\lambda},\beta^{\lambda})\) satisfies the first order conditions
	\begin{equation*}
		\left\{	\begin{aligned}
			a_{i} & = \int \phi_{\lambda}(x,\alpha^{\lambda})^{1-\lambda} \psi_{\lambda}(y,\beta^{\lambda})  \exp(\lambda \langle x, x_{i} \rangle +  \lambda \alpha_{i}^{\lambda}) \tilde{\gamma}(\md x,\md y), \\
			b_{j} & = \int \phi_{\lambda}(x,\alpha^{\lambda}) \psi_{\lambda}(y,\beta^{\lambda})^{1-\lambda}  \exp(\lambda \langle y, y_{j} \rangle +  \lambda \beta_{j}^{\lambda}) \tilde{\gamma}(\md x,\md y).
		\end{aligned}
		\right.
	\end{equation*}
	Therefore, for any converging subsequence of \(\{(\alpha^{\lambda},\beta^{\lambda})\}_{\lambda\geq 2}\), its limit \((\alpha^{\infty},\beta^{\infty})\) satisfies
	\begin{equation*}
		\left\{	\begin{aligned}
			a_{i} & = \int_{A_{i}(\alpha^{\infty})\times \mathcal{Y}} \exp(f(x,\alpha^{\infty}) +g(y,\beta^{\infty}))  \tilde{\gamma}(\md x,\md y), \\
			b_{j} & = \int_{\mathcal{X}\times B_{j}(\beta^{\infty})}\exp( f(x,\alpha^{\infty}) +g(y,\beta^{\infty}))\tilde{\gamma}(\md x,\md y),
		\end{aligned}
		\right.
	\end{equation*}
	which further implies \((\alpha^{\infty},\beta^{\infty})\) is an optimizer of the dual problem.
	By the uniqueness of the dual problem in Theorem \ref{thm-semi}, we must have \((\alpha^{\infty},\beta^{\infty})\sim_{\oplus}(\alpha^{*},\beta^{*})\).
	Hence, we deduce \(\lim_{\lambda\to\infty}\|(\alpha^{\lambda},\beta^{\lambda})-(\alpha^{*},\beta^{*})\|_{l^{\infty}_{\oplus}}=0\).
\end{proof}

\section{Numerical experiments}
\label{sec-exam}
In this section, we numerically verify the blow-up phenomenon in Theorem \ref{thm-blowup}, and numerically verify the linear convergence rate of both gradient ascent method and sinkhorn method.

Throughout this section, we work with a synthetic experiment setting in dimension
$d=2$ and consider discrete marginals with $n=10$ and $m=10$. The support points are generated once and then fixed throughout: $x_i$ are sampled uniformly from $[-1,1]^2$ and $y_i$ are sampled uniformly from $[-2,1]^2$. The reference coupling $\gamma$ is correlated Gaussian defined by $(X,Y)\sim \gamma$ with $X=\mu_0+Z_0$, $Y=\mu_1+\rho Z_0+\sqrt{1-\rho^2}Z_1$, $\mu_0=(1,-0.5)$, $\mu_1=(-1,0.8)$, $\rho=-0.4$, $Z_0,Z_1\sim N(0,I_d)$ independently.

\subsection{Blow-up phenomenon}

\paragraph{Set-up.}
For each $\varepsilon\in\{1,4^{-1},\dots,4^{-7}\}$, we approximately solve the dual problem $U_\varepsilon$ in \eqref{eqn-dual-blowup} using gradient ascent, where the integrals against $\gamma$ are evaluated by Monte Carlo sampling with batch size $8000$. To achieve high accuracy when $\varepsilon$ is small, we run $T=1000$ iterations for each $\varepsilon$. We then compute $I_\varepsilon(\pi^{*,\varepsilon})$ from the final dual iterates using the identity \eqref{I_epsilon} in the proof of Theorem~\ref{thm-blowup}.

\begin{figure}[H]
	\centering
	\includegraphics[width=0.5\linewidth]{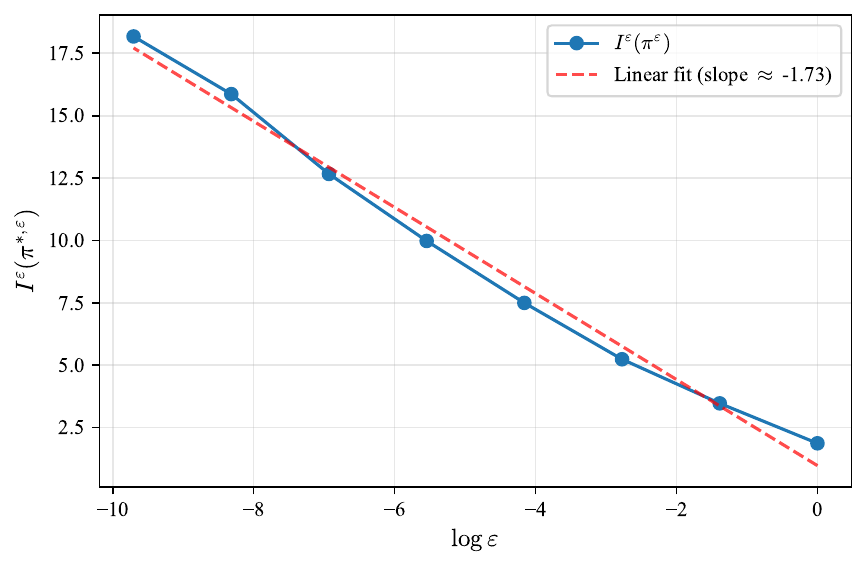}
	\caption{Scaling of $I^\varepsilon(\pi^{*,\varepsilon})$ with respect to $\varepsilon$.}
	\label{fig:I_eps}
\end{figure}
\paragraph{Results.} Figure~\ref{fig:I_eps} plots $I_\varepsilon(\pi^{*,\varepsilon})$ against $\log \varepsilon$. The curve is close to linear, confirming the logarithmic blow-up stated in Theorem~\ref{thm-blowup}.
In particular, in our $d=2$ experiment, the slope with respect to $\log(\varepsilon)$ is close to $-2$, confirming the coefficient $-d$ in the leading term in Theorem~\ref{thm-blowup}.

\subsection{Gradient ascent method}
\label{subsec:gradient-ascent}

\paragraph{Set-up.} %
We perform the gradient ascent algorithm \eqref{eq:gradient-ascent} with learning rate $\eta = 0.01$ and run the update for $T=2000$ iterations.

\paragraph{Results.} We visualize the convergence of $\|(\alpha^t,\beta^t)-(\alpha^T,\beta^T)\|_{l_{\oplus}^2}$ across iterations $t$, using them as proxies for $\|(\alpha^t,\beta^t)-(\alpha^*,\beta^*)\|_{l_{\oplus}^2}$.
The log-scale plots of the $y$-axis exhibit a clear linear decay, indicating a linear convergence rate.
This behavior is consistent with the convergence rate established in Proposition~\ref{prop-ga}.

\begin{figure}[H]
	\centering
	\includegraphics[width=0.5\linewidth]{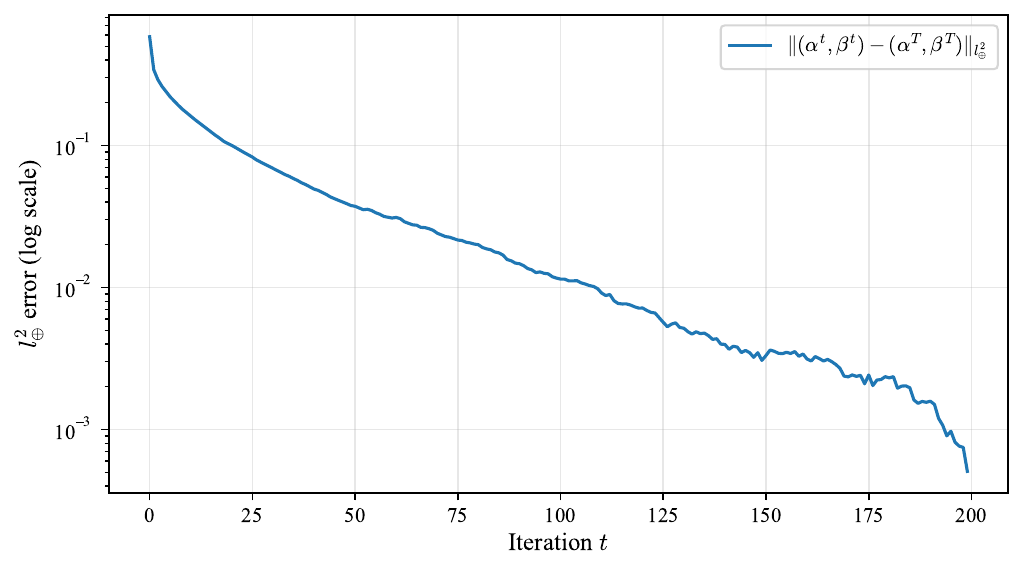}
	\caption{Convergence of $\|(\alpha^t,\beta^t)-(\alpha^T,\beta^T)\|_{l_{\oplus}^2}$.}
	\label{fig:GA_theta_convergence}
\end{figure}

Beyond convergence, we visualize the geometric structure induced by the learned dual variables. Given the final iterates $(\alpha^T,\beta^T)$, we consider the associated max-affine potentials
\begin{equation*}
	f_{\alpha^T}(x) = \max_{1\leq i\leq n}\{ \langle x,x_i \rangle +\alpha_i^T\}, \quad g_{\beta^T}(y) = \max_{1\leq j\leq m}\{ \langle y,y_j \rangle +\beta_j^T\},
\end{equation*}
which induce polyhedral partitions of the $x-$ and $y-$ spaces.

Using a fine grid over $[-3,3]^2$, we compute the active affine index at each grid point and visualize the resulting partitions together with samples drawn from the reference measure
$\gamma$.
The resulting Figure \ref{fig:GA_partition} shows well-defined regions of dominance for each affine component, with empirical mass concentrating in a subset of active cells.

\begin{figure}[H]
	\centering
	\includegraphics[width=0.9\linewidth]{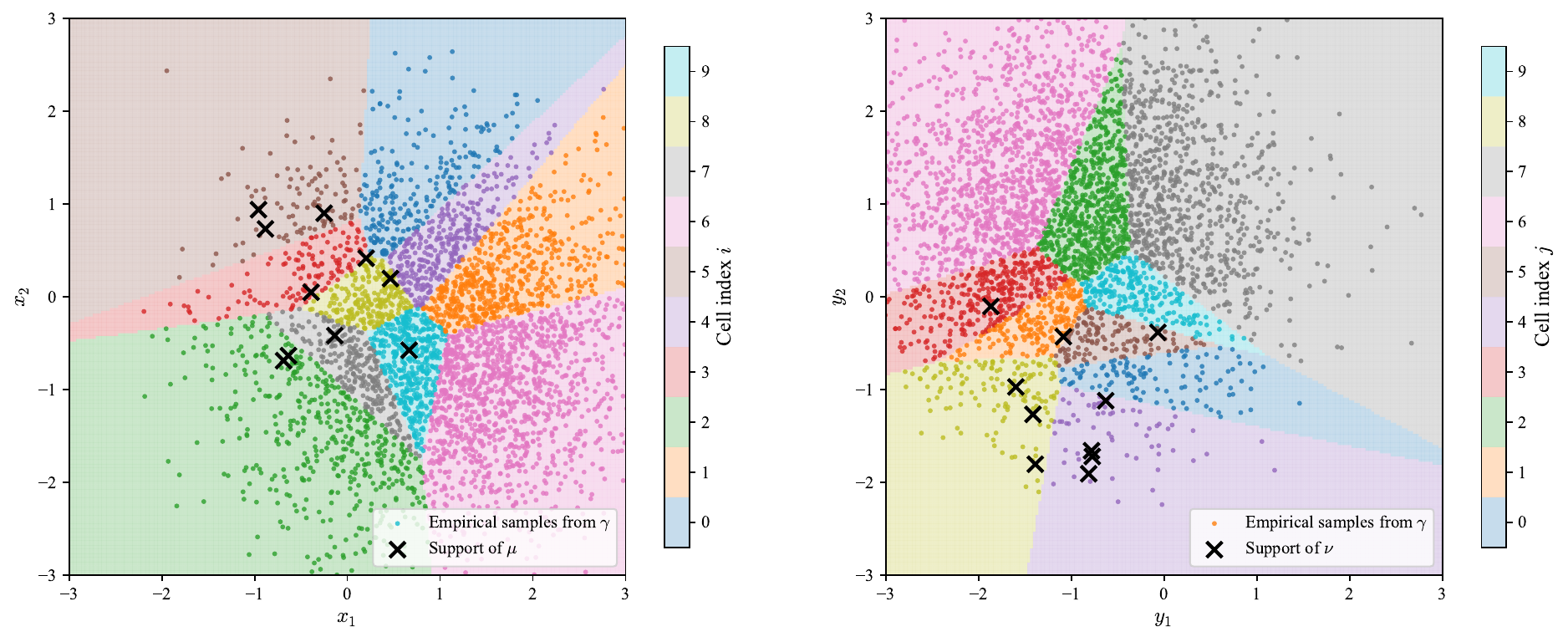}
	\caption{Space partition and empirical mass.}\label{fig:GA_partition}
\end{figure}
Figure \ref{fig:GA_partition} illustrates that the final dual iterates define clear max-affine (polyhedral) partitions of the
$x$- and $y$-spaces; the overlaid samples concentrate in only a subset of cells, while other affine pieces are active only in low-density regions and therefore have negligible influence on the learned solution.

\subsection{Sinkhorn-type method}
\paragraph{Set-up.} Note that we fix the  same data set as in Section \ref{subsec:gradient-ascent} for the comparison of state space partition. We take the regularization parameter as $\lambda=10$ and run the algorithm for $T=500$ iterations.

\paragraph{Results.} Similar as Section \ref{subsec:gradient-ascent}, we visualize the convergence of $\|(\alpha^t,\beta^t)-(\alpha^T,\beta^T)\|_{l_{\oplus}^\infty}$ across iterations $t$, using them as proxies for $\|(\alpha^t,\beta^t)-(\alpha^*,\beta^*)\|_{l_{\oplus}^\infty}$.
The log-scale plots of the $y$-axis exhibit a clear linear decay for iterations before $400$, indicating a linear convergence rate.
This behavior is consistent with the convergence rate established in Proposition~\ref{prop-sinkhorn}.
We believe that after 30 steps the fluctuation from the Monte Carlo sampling dominates the algorithm convergence.

\begin{figure}[H]
	\centering
	\includegraphics[width=0.5\linewidth]{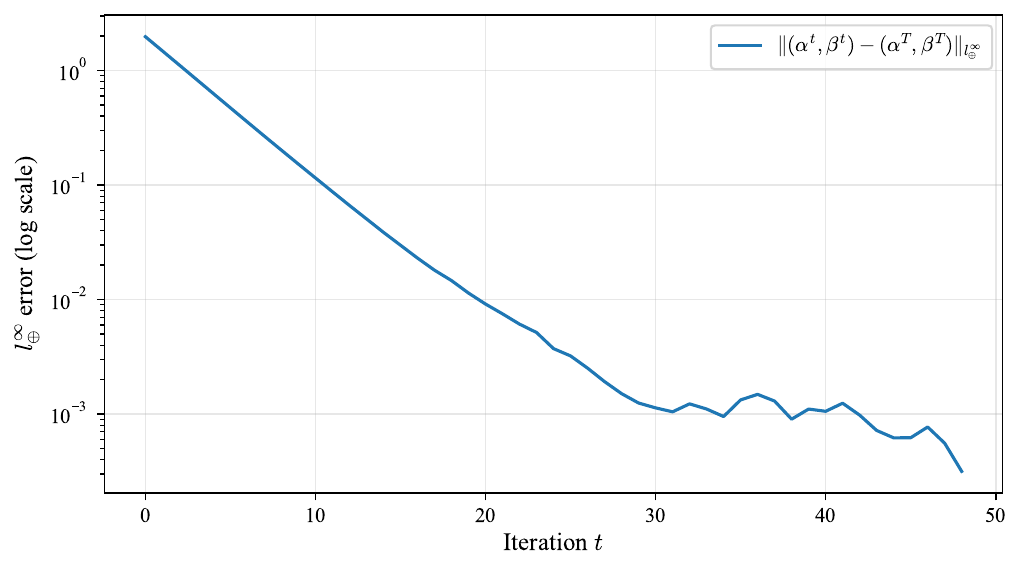}
	\caption{Convergence of $\|(\alpha^t,\beta^t)-(\alpha^T,\beta^T)\|_{l_{\oplus}^\infty}$.}
	\label{fig:SIN_theta_convergence}
\end{figure}

Despite the presence of an regularization term in the Sinkhorn formulation, the resulting partitions (see Figure \ref{fig:SIN_partition}) closely track those obtained from the unregularized gradient ascent method (see Figure \ref{fig:GA_partition}). In particular, even with a finite regularization parameter $\lambda$, the max-affine structure induced by the Sinkhorn dual variables  yields cell boundaries that are nearly identical to those of the limiting unregularized solution. This indicates that the entropic smoothing primarily stabilizes the optimization and accelerates convergence, while introducing only negligible geometric bias in the learned partition. The close agreement between the two methods suggests that Sinkhorn regularization preserves the essential structural features of the optimal dual potentials, making it a promising and computationally robust alternative to unregularized gradient ascent for recovering interpretable space partitions.

\begin{figure}[H]
	\centering
	\includegraphics[width=0.9\linewidth]{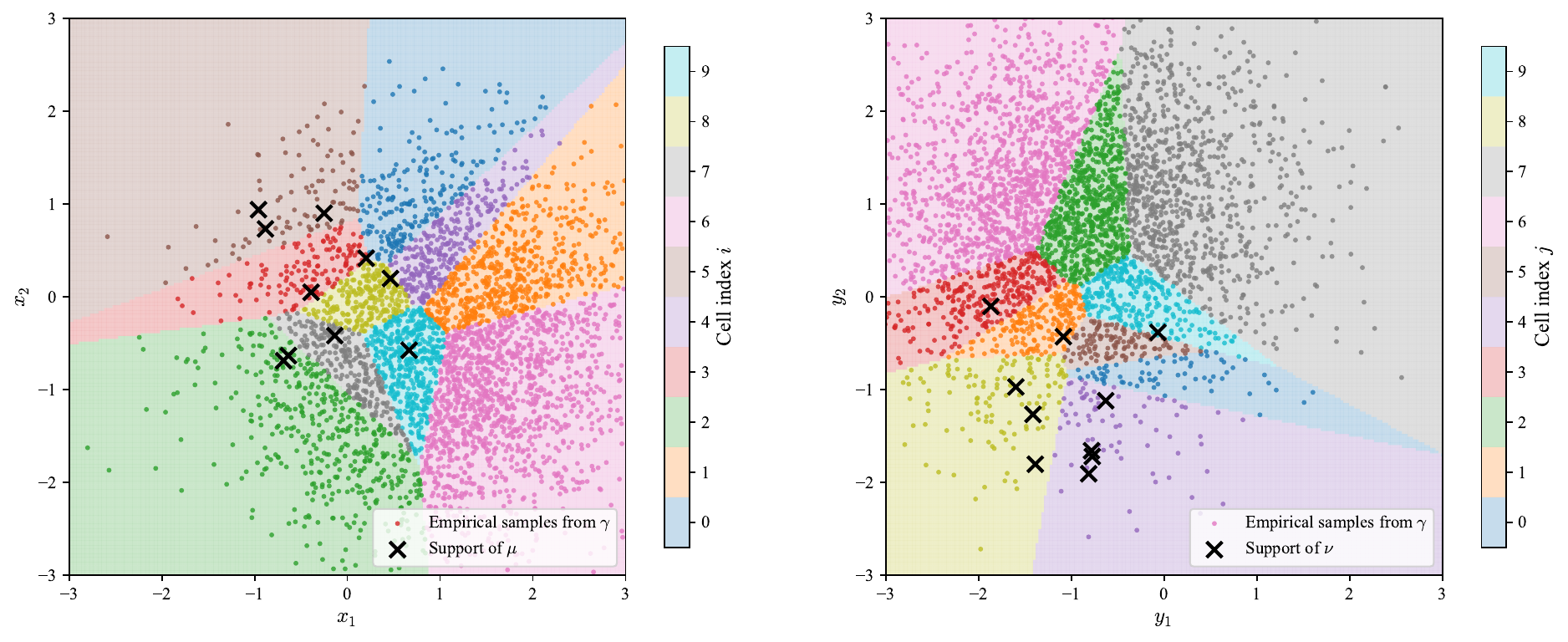}
	\caption{Space partition and empirical mass.}\label{fig:SIN_partition}
\end{figure}

\bibliography{reference}

\begin{thebibliography}{46}
\providecommand{\natexlab}[1]{#1}
\providecommand{\url}[1]{\texttt{#1}}
\expandafter\ifx\csname urlstyle\endcsname\relax
  \providecommand{\doi}[1]{doi: #1}\else
  \providecommand{\doi}{doi: \begingroup \urlstyle{rm}\Url}\fi

\bibitem[Alouadi et~al.(2026)Alouadi, Henry-Labordère, Loeper, Mazhar, Pham, and Touzi]{alouadi2026lightsbb}
A.~Alouadi, P.~Henry-Labordère, G.~Loeper, O.~Mazhar, H.~Pham, and N.~Touzi.
\newblock {LightSBB-M}: Bridging {S}chr{\"o}dinger and bass for generative diffusion modeling, 2026.
\newblock arXiv:2601.19312.

\bibitem[Aryan and Ghosal(2025)]{aryan25entropic}
S.~Aryan and P.~Ghosal.
\newblock Entropic selection principle for {{Monge}}'s optimal transport, Aug. 2025.
\newblock arXiv:2502.16370.

\bibitem[Bai et~al.(2023)Bai, He, Jiang, and Oblój]{bai23wasserstein}
X.~Bai, G.~He, Y.~Jiang, and J.~Oblój.
\newblock Wasserstein distributional robustness of neural networks.
\newblock \emph{Advances in Neural Information Processing Systems}, 36, 2023.

\bibitem[Bai et~al.(2025)Bai, He, Jiang, and Oblój]{bai25wasserstein}
X.~Bai, G.~He, Y.~Jiang, and J.~Oblój.
\newblock Wasserstein distributional adversarial training for deep neural networks, Feb. 2025.
\newblock arXiv:2502.09352.

\bibitem[Bartl and Wiesel(2023)]{bartl23sensitivity}
D.~Bartl and J.~Wiesel.
\newblock Sensitivity of multiperiod optimization problems with respect to the adapted {Wasserstein} distance.
\newblock \emph{SIAM Journal on Financial Mathematics}, 14\penalty0 (2):\penalty0 704--720, June 2023.

\bibitem[Bartl et~al.(2021)Bartl, Drapeau, Obłój, and Wiesel]{bartl21sensitivity}
D.~Bartl, S.~Drapeau, J.~Obłój, and J.~Wiesel.
\newblock Sensitivity analysis of {{Wasserstein}} distributionally robust optimization problems.
\newblock \emph{Proceedings of the Royal Society A: Mathematical, Physical and Engineering Sciences}, 477\penalty0 (2256):\penalty0 20210176, Dec. 2021.

\bibitem[Bertsekas and Shreve(1996)]{bertsekas1996stochastic}
D.~Bertsekas and S.~E. Shreve.
\newblock \emph{Stochastic {{Optimal Control}}: {{The Discrete-Time Case}}}.
\newblock {Athena Scientific}, Dec. 1996.

\bibitem[Blanchet and Murthy(2019)]{blanchet19quantifying}
J.~Blanchet and K.~Murthy.
\newblock Quantifying distributional model risk via optimal transport.
\newblock \emph{Mathematics of Operations Research}, 44\penalty0 (2):\penalty0 565--600, May 2019.

\bibitem[Carlier et~al.(2017)Carlier, Duval, Peyr{\'e}, and Schmitzer]{carlier2017convergence}
G.~Carlier, V.~Duval, G.~Peyr{\'e}, and B.~Schmitzer.
\newblock Convergence of entropic schemes for optimal transport and gradient flows.
\newblock \emph{SIAM Journal on Mathematical Analysis}, 49\penalty0 (2):\penalty0 1385--1418, 2017.

\bibitem[Chen et~al.(2023{\natexlab{a}})Chen, Huang, Zhao, and Wang]{chen2023score}
M.~Chen, K.~Huang, T.~Zhao, and M.~Wang.
\newblock Score approximation, estimation and distribution recovery of diffusion models on low-dimensional data.
\newblock In \emph{{International Conference on Machine Learning}}, pages 4672--4712. PMLR, 2023{\natexlab{a}}.

\bibitem[Chen et~al.(2024)Chen, Mei, Fan, and Wang]{chen2024overview}
M.~Chen, S.~Mei, J.~Fan, and M.~Wang.
\newblock An overview of diffusion models: Applications, guided generation, statistical rates and optimization, 2024.
\newblock arXiv:2404.07771.

\bibitem[Chen et~al.(2023{\natexlab{b}})Chen, Liu, and Theodorou]{chen2021likelihood}
T.~Chen, G.-H. Liu, and E.~A. Theodorou.
\newblock Likelihood training of {Schr\"odinger} bridge using forward-backward {SDEs} theory, 2023{\natexlab{b}}.
\newblock arXiv:2110.11291.

\bibitem[Chen et~al.(2016)Chen, Georgiou, and Pavon]{chen2016relation}
Y.~Chen, T.~T. Georgiou, and M.~Pavon.
\newblock On the relation between optimal transport and {Schr\"odinger} bridges: A stochastic control viewpoint.
\newblock \emph{Journal of Optimization Theory and Applications}, 169\penalty0 (2):\penalty0 671--691, 2016.

\bibitem[Cuturi(2013)]{cuturi2013sinkhorn}
M.~Cuturi.
\newblock Sinkhorn distances: Lightspeed computation of optimal transport.
\newblock \emph{Advances in Neural Information Processing Systems}, 26, 2013.

\bibitem[Dai~Pra(1991)]{dai1991stochastic}
P.~Dai~Pra.
\newblock A stochastic control approach to reciprocal diffusion processes.
\newblock \emph{Applied Mathematics and Optimization}, 23\penalty0 (1):\penalty0 313--329, 1991.

\bibitem[De~Bortoli et~al.(2021)De~Bortoli, Thornton, Heng, and Doucet]{de2021diffusion}
V.~De~Bortoli, J.~Thornton, J.~Heng, and A.~Doucet.
\newblock Diffusion {S}chr{\"o}dinger bridge with applications to score-based generative modeling.
\newblock \emph{Advances in Neural Information Processing Systems}, 34:\penalty0 17695--17709, 2021.

\bibitem[Deming and Stephan(1940)]{deming1940least}
W.~E. Deming and F.~F. Stephan.
\newblock On a least squares adjustment of a sampled frequency table when the expected marginal totals are known.
\newblock \emph{The Annals of Mathematical Statistics}, 11\penalty0 (4):\penalty0 427--444, 1940.

\bibitem[Di~Marino and Louet(2018)]{marino18entropic}
S.~Di~Marino and J.~Louet.
\newblock The entropic regularization of the {Monge} problem on the real line.
\newblock \emph{SIAM Journal on Mathematical Analysis}, 50\penalty0 (4):\penalty0 3451--3477, 2018.

\bibitem[Gangbo and McCann(1996)]{gangbo1996geometry}
W.~Gangbo and R.~J. McCann.
\newblock The geometry of optimal transportation.
\newblock \emph{Acta Mathematica}, 177\penalty0 (2):\penalty0 113--161, 1996.

\bibitem[Garg et~al.(2024)Garg, Zhang, and Zhou]{garg24soft}
J.~Garg, X.~Zhang, and Q.~Zhou.
\newblock Soft-constrained {Schr{\"o}dinger} bridge: a stochastic control approach.
\newblock In \emph{International Conference on Artificial Intelligence and Statistics}, pages 4429--4437. PMLR, 2024.

\bibitem[González-Sanz et~al.(2025)González-Sanz, Nutz, and Valdevenito]{gonzalezsanz2025linear}
A.~González-Sanz, M.~Nutz, and A.~R. Valdevenito.
\newblock Linear convergence of gradient descent for quadratically regularized optimal transport, 2025.
\newblock arXiv:2509.08547.

\bibitem[Hamdouche et~al.(2023)Hamdouche, Henry-Labordere, and Pham]{hamdouche2023generative}
M.~Hamdouche, P.~Henry-Labordere, and H.~Pham.
\newblock Generative modeling for time series via {S}chr{\"o}dinger bridge, 2023.
\newblock arXiv:2304.05093.

\bibitem[Han et~al.(2024)Han, Razaviyayn, and Xu]{han2024neural}
Y.~Han, M.~Razaviyayn, and R.~Xu.
\newblock Neural network-based score estimation in diffusion models: Optimization and generalization, 2024.
\newblock arXiv:2401.15604.

\bibitem[Ho et~al.(2020)Ho, Jain, and Abbeel]{ho2020denoising}
J.~Ho, A.~Jain, and P.~Abbeel.
\newblock Denoising diffusion probabilistic models.
\newblock \emph{Advances in Neural Information Processing Systems}, 33:\penalty0 6840--6851, 2020.

\bibitem[Jiang(2024)]{jiang24Duality}
Y.~Jiang.
\newblock Duality of causal distributionally robust optimization, Jan. 2024.
\newblock arXiv:2401.16556.

\bibitem[Jiang and Obłój(2025)]{jiang25sensitivity}
Y.~Jiang and J.~Obłój.
\newblock Sensitivity of causal distributionally robust optimization, May 2025.
\newblock arXiv:2408.17109.

\bibitem[Lai et~al.(2025)Lai, Song, Kim, Mitsufuji, and Ermon]{lai2025principles}
C.-H. Lai, Y.~Song, D.~Kim, Y.~Mitsufuji, and S.~Ermon.
\newblock The principles of diffusion models, 2025.

\bibitem[Ley(2025)]{ley25entropic}
A.~Ley.
\newblock Entropic selection for optimal transport on the line with distance cost, Dec. 2025.
\newblock arXiv:2512.05282.

\bibitem[Léonard(2013)]{leonard2013survey}
C.~Léonard.
\newblock A survey of the {{Schrödinger}} problem and some of its connections with optimal transport.
\newblock \emph{Discrete and Continuous Dynamical Systems}, 34\penalty0 (4):\penalty0 1533--1574, 2013.

\bibitem[Ma et~al.(2025)Ma, Tan, and Xu]{ma2025schr}
J.~Ma, Y.~Tan, and R.~Xu.
\newblock {S}chr{\"o}dinger bridge for generative {AI}: Soft-constrained formulation and convergence analysis.
\newblock \emph{arXiv preprint arXiv:2510.11829}, 2025.

\bibitem[Nutz(2021)]{nutz2021introduction}
M.~Nutz.
\newblock Introduction to entropic optimal transport.
\newblock \emph{Lecture notes, Columbia University}, 2021.

\bibitem[Nutz(2025)]{nutz25quadratically}
M.~Nutz.
\newblock Quadratically regularized optimal transport: Existence and multiplicity of potentials.
\newblock \emph{SIAM Journal on Mathematical Analysis}, 57\penalty0 (3):\penalty0 2622--2649, 2025.

\bibitem[Nutz and Wiesel(2022)]{nutz2022entropic}
M.~Nutz and J.~Wiesel.
\newblock Entropic optimal transport: Convergence of potentials.
\newblock \emph{Probability Theory and Related Fields}, 184\penalty0 (1):\penalty0 401--424, 2022.

\bibitem[Peluchetti(2023)]{peluchetti2023diffusion}
S.~Peluchetti.
\newblock Diffusion bridge mixture transports, {S}chr{\"o}dinger bridge problems and generative modeling.
\newblock \emph{Journal of Machine Learning Research}, 24\penalty0 (374):\penalty0 1--51, 2023.

\bibitem[Peyr{\'e} et~al.(2019)Peyr{\'e}, Cuturi, et~al.]{peyre2019computational}
G.~Peyr{\'e}, M.~Cuturi, et~al.
\newblock Computational optimal transport: With applications to data science.
\newblock \emph{Foundations and Trends in Machine Learning}, 11\penalty0 (5-6):\penalty0 355--607, 2019.

\bibitem[Rigollet and Weed(2018)]{rigollet2018entropic}
P.~Rigollet and J.~Weed.
\newblock Entropic optimal transport is maximum-likelihood deconvolution.
\newblock \emph{Comptes Rendus. Math{\'e}matique}, 356\penalty0 (11-12):\penalty0 1228--1235, 2018.

\bibitem[Sauldubois and Touzi(2024)]{sauldubois24first}
N.~Sauldubois and N.~Touzi.
\newblock First order martingale model risk and semi-static hedging, Oct. 2024.
\newblock arXiv:2410.06906.

\bibitem[Schr{\"o}dinger(1931)]{schrodinger1931uber}
E.~Schr{\"o}dinger.
\newblock {\"U}ber die umkehrung der naturgesetze.
\newblock \emph{Sitzungsberichte der Preu{{\ss}}ischen Akademie der Wissenschaften, Physikalisch-Mathematische Klasse}, pages 144--153, 1931.

\bibitem[Schr{\"o}dinger(1932)]{schrodinger1932sur}
E.~Schr{\"o}dinger.
\newblock Sur la th{\'e}orie relativiste de l'electron et l'interpretation de la m{\'e}canique quantique.
\newblock \emph{Annales de l'Institut Henri Poincar{\'e}}, 2:\penalty0 269--310, 1932.

\bibitem[Shi et~al.(2023)Shi, De~Bortoli, Campbell, and Doucet]{shi2024diffusion}
Y.~Shi, V.~De~Bortoli, A.~Campbell, and A.~Doucet.
\newblock Diffusion {S}chr{\"o}dinger bridge matching.
\newblock \emph{Advances in Neural Information Processing Systems}, 36, 2023.

\bibitem[Sinha et~al.(2018)Sinha, Namkoong, and Duchi]{sinha18certifying}
A.~Sinha, H.~Namkoong, and J.~Duchi.
\newblock Certifying some distributional robustness with principled adversarial training.
\newblock In \emph{International {{Conference}} on {{Learning Representations}}}, 2018.

\bibitem[Song et~al.(2021)Song, Sohl-Dickstein, Kingma, Kumar, Ermon, and Poole]{song2020score}
Y.~Song, J.~Sohl-Dickstein, D.~P. Kingma, A.~Kumar, S.~Ermon, and B.~Poole.
\newblock Score-based generative modeling through stochastic differential equations.
\newblock In \emph{International Conference on Learning Representations}, 2021.

\bibitem[Séjourné et~al.(2023)Séjourné, Peyré, and Vialard]{sejourne23unblanced}
T.~Séjourné, G.~Peyré, and F.-X. Vialard.
\newblock Unbalanced optimal transport, from theory to numerics.
\newblock In E.~Trélat and E.~Zuazua, editors, \emph{Numerical Control: Part B}, volume~24 of \emph{Handbook of Numerical Analysis}, pages 407--471. Elsevier, 2023.

\bibitem[Vargas et~al.(2021)Vargas, Thodoroff, Lamacraft, and Lawrence]{vargas2021solving}
F.~Vargas, P.~Thodoroff, A.~Lamacraft, and N.~Lawrence.
\newblock Solving {S}chr{\"o}dinger bridges via maximum likelihood.
\newblock \emph{Entropy}, 23\penalty0 (9):\penalty0 1134, 2021.

\bibitem[Wang et~al.(2021)Wang, Jiao, Xu, Wang, and Yang]{wang2021deep}
G.~Wang, Y.~Jiao, Q.~Xu, Y.~Wang, and C.~Yang.
\newblock Deep generative learning via {S}chr{\"o}dinger bridge.
\newblock In \emph{International Conference on Machine Learning}, pages 10794--10804. PMLR, 2021.

\bibitem[Zhang et~al.(2024)Zhang, Yang, and Gao]{zhang24short}
L.~Zhang, J.~Yang, and R.~Gao.
\newblock A short and general duality proof for wasserstein distributionally robust optimization.
\newblock \emph{Operations Research}, pages 1723--2295, July 2024.

\end{thebibliography}

\end{document}